\documentclass{article}
\usepackage{amssymb,amsmath,amsthm,latexsym, enumerate}
\usepackage{latexsym}
\newtheorem{thm}{Theorem}[section]
\newtheorem{proposition}{Proposition}[section]
\newtheorem{definition}{Definition}[section]
\newtheorem{corollary}{Corollary}[section]
\newtheorem{lemma}{Lemma}[section]

\newcommand{\rad}{\operatorname{rad}}

\begin{document}

\title{Cell algebra structure on generalized Schur algebras}

\author{ Robert May\\
Department of Mathematics and Computer Science\\
Longwood University\\
201 High Street, Farmville, VA 23909\\
rmay@longwood.edu }
\maketitle

\section {Introduction}

     A family of ``generalized Schur algebras'' were first introduced in \cite{MA} and \cite{May}.  In \cite{May2} the left and right generalized Schur algebras were shown to be ``double coset algebras''.  In \cite{May} and \cite{May4} a stratification of these algebras was given leading to a parameterization, in most cases, of their irreducible representations.  In this paper we obtain cell algebra structures for these algebras in the sense of \cite{May3}.  (The Cell algebras of \cite{May3}  coincide with the standardly based algebras previously introduceed by Du and Rui in \cite{DuRui}.)  The properties of cell algebras combined with the parameterization of the irreducible representations leads to a more concrete description of all these irreducibles.  In certain cases these algebras are shown to be quasi-hereditary.  
     
In section 2 we review the definition and properties of cell algebras as presented in \cite{May3}.  In section 3 we describe the cell bases found in \cite{May3} for a family of semigroups including the full transformation semigroups $\mathcal{T}_r$  and the rook semigroups $\Re_r$.  In section 4 we give cell bases for the left and right generalized Schur algebras corresponding to these semigroups.  Finally, in section 5 we use the cell algebra structure to describe the irreducible representations of these algebras and to determine when they are quasi-hereditary.

\section {Cell algebra structures}  

In this section we review, without proofs, the definition and properties of cell algebras as presented in \cite{May3}.  (These algebras were previously studied as ``standardly based algebras'' by Du and Rui in \cite{DuRui}.)   Let $R$ be a commutative integral domain with unit 1 and let $A$ be an associative, unital $R$-algebra.  Let $\Lambda $ be a finite set with a partial order $ \leqslant $ and for each $\lambda  \in \Lambda $ let $L\left( \lambda  \right),R\left( \lambda  \right)$ be finite sets of ``left indices'' and ``right indices''. Assume that for each $\lambda  \in \Lambda ,s \in L\left( \lambda  \right),{\text{ and }}t \in R\left( \lambda  \right)$ there is an element $\,_s C_t ^\lambda   \in A$ such that the map $\left( {\lambda ,s,t} \right) \mapsto \,_s C_t ^\lambda  $ is injective and
\[ C = \left\{ {_s C_t ^\lambda  :\lambda  \in \Lambda ,s \in L(\lambda ),t \in R\left( \lambda  \right)} \right\} \]
is a free $R$-basis for $A$.  Define $R$-submodules of $A$ by 
\[A^\lambda   = R{\text{ - span of }}\left\{ {_s C_t ^\mu  :\mu  \in \Lambda ,\mu  \geqslant \lambda ,s \in L(\mu ),t \in R\left( \mu  \right)} \right\}\]
 and
\[ \hat A^\lambda   = R{\text{ - span of }}\left\{ {_s C_t ^\mu  :\mu   \in \Lambda ,\mu  > \lambda ,s \in L(\mu ),t \in R\left( \mu  \right)} \right\} . \]

\begin {definition} \label {d2.1}  Given $A,\Lambda,C$, $A$ is a cell algebra with poset $\Lambda$ and cell basis $C$ if 
    \begin{description}
     \item[i] For any $a \in A,\lambda  \in \Lambda ,{\text{ and }}s,s' \in L\left( \lambda  \right)$, there exists $r_L  = r_L \left( {a,\lambda ,s,s'} \right) \in R$ such that, for any $t \in R\left( \lambda  \right)$, $a \cdot \,_s C_t ^\lambda   = \sum\limits_{s' \in L\left( \lambda  \right)} {r_L  \cdot \,_{s'} C_t ^\lambda  } \, \, \bmod \hat A^\lambda  $, and
      \item[ii] For any $a \in A,\lambda  \in \Lambda ,{\text{ and }}t,t' \in R\left( \lambda  \right)$, there exists  $r_R  = r_R \left( {a,\lambda ,t,t'} \right) \in R$ such that, for any $s \in L\left( \lambda  \right)$, $_s C_t ^\lambda   \cdot a = \sum\limits_{t' \in R\left( \lambda  \right)} {r_R  \cdot \,_s C_{t'} ^\lambda  } \, \,  \bmod \hat A^\lambda  $ .
     \end{description}
\end {definition}

     Consider a fixed cell algebra $A$ with poset $\Lambda$ and cell basis $C$.

\begin {lemma} \label {l2.1} 
    \text{   } 
    \begin{enumerate} [\upshape(a)]
       \item  $A^\lambda  $ and $\hat A^\lambda  $ are two sided ideals in $A$ for any $\lambda  \in \Lambda $.
       \item  For $\lambda  \in \Lambda ,t,t' \in R\left( \lambda  \right),s,s' \in L\left( \lambda  \right)$,  $r_L \left( {_{s'} C_t ^\lambda  ,\lambda ,s,s'} \right) = r_R \left( {_s C_{t'} ^\lambda  ,\lambda ,t,t'} \right)$.                     
        \item Given $\lambda  \in \Lambda ,t \in R\left( \lambda  \right),s \in L\left( \lambda  \right)$, there exists $r_{st}  \in R$ such that for any $s' \in L\left( \lambda  \right),t' \in R\left( \lambda  \right)$ we have $_{s'} C_t ^\lambda  \,\,_s C_{t'} ^\lambda   = r_{st} \,\,_{s'} C_{t'} ^\lambda  \bmod \hat A^\lambda  $. In fact  	
$r_{st}  = r_L \left( {_{s'} C_t ^\lambda  ,\lambda ,s,s'} \right) = r_R \left( {_s C_{t'} ^\lambda  ,\lambda ,t,t'} \right)$.
     \end{enumerate}
\end {lemma}

     By lemma \ref {l2.1}, part (a), $A/\hat A^\lambda  $ is a unital $R$-algebra and $A^\lambda  /\hat A^\lambda  $ is a two sided ideal in $A/\hat A^\lambda  $.  Observe that as an $R$-module $A^\lambda  /\hat A^\lambda  $ is free with a basis $\left\{ {_s C_t ^\lambda   + \hat A^\lambda  :s \in L\left( \lambda  \right),t \in R\left( \lambda  \right)} \right\}$.

     For a fixed $t \in R\left( \lambda  \right)$, define $\,_L C_t ^\lambda $ as the free $R$-submodule of $A^\lambda  /\hat A^\lambda $ with basis $\left\{ {_s C_t ^\lambda   + \hat A^\lambda  :s \in L\left( \lambda  \right)} \right\}$.  By property (i), $\,_L C_t ^\lambda  $ is a left $A$-module and $\,_L C_t ^\lambda   \cong \,_L C_{t'} ^\lambda $ as left $A$-modules for any $t,t' \in R\left( \lambda  \right)$.  Evidently, as left $A$-modules we have $A^\lambda  /\hat A^\lambda   \cong \mathop  \bigoplus \limits_{t \in R(\lambda )} \,\,_L C_t ^\lambda  $.  

\begin {definition} \label {d2.2}  The left cell module for $\lambda $ is the left $A$-module $\,_L C^\lambda  $ defined as follows:  Take the free $R$-module with a basis $\left\{ {_s C^\lambda  :s \in L\left( \lambda  \right)} \right\}$ and define the left action of $A$ by $a \cdot \,_s C^\lambda   = \sum\limits_{s' \in L\left( \lambda  \right)} {r_L \left( {a,\lambda ,s,s'} \right)} \,\,_{s'} C^\lambda  $ for $a \in A$. 
\end{definition}

     For any $t \in R\left( \lambda  \right)$, $\,_s C^\lambda   \mapsto \,_s C_t ^\lambda   + \hat A^\lambda  $ gives a left $A$-module isomorphism $\phi _t :\,_L C^\lambda   \to \,_L C_t ^\lambda . $  Then $A^\lambda  /\hat A^\lambda   \cong \mathop  \bigoplus \limits_{t \in R(\lambda )} \,\,_L C_t ^\lambda  $ is isomorphic to the direct sum of $\left| {R\left( \lambda  \right)} \right|$ copies of $\,_L C^\lambda  $.

      In a parallel way, for a fixed $s \in L\left( \lambda  \right)$, define $\,_s C_R ^\lambda  $ as the free $R$-module with basis $\left\{ {_s C_t ^\lambda   + \hat A^\lambda  :t \in R\left( \lambda  \right)} \right\}$, an $R$-submodule of $A^\lambda  /\hat A^\lambda  $.  By property (ii), $\,_s C_R ^\lambda  $ is a right $A$-module and $\,_s C_R ^\lambda   \cong \,_{s'} C_R ^\lambda  $ as right $A$-modules for any $s,s' \in L\left( \lambda  \right)$.  As right $A$-modules we have $A^\lambda  /\hat A^\lambda   \cong \mathop  \bigoplus \limits_{s \in L\left( \lambda  \right)} \,\,_s C_R ^\lambda  $.  

\begin {definition} \label {d2.3}  The right cell module for $\lambda $ is the right $A$-module $C_R ^\lambda  $ defined as follows:  Take the free $R$-module with a basis $\left\{ {C_t ^\lambda  :t \in R\left( \lambda  \right)} \right\}$ and define the right action of $A$ by $C_t ^\lambda   \cdot a = \sum\limits_{t' \in R\left( \lambda  \right)} {r_R \left( {a,\lambda ,t,t'} \right)} \,C_{t'} ^\lambda  $ for $a \in A$. 
\end{definition}

     For any $s \in L\left( \lambda  \right)$, $C_t ^\lambda   \mapsto \,_s C_t ^\lambda   + \hat A^\lambda  $ gives a right $A$-module isomorphism $_s \phi :C_R ^\lambda   \to \,_s C_R ^\lambda  .$  Then $A^\lambda  /\hat A^\lambda   \cong \mathop  \bigoplus \limits_{s \in L(\lambda )} \,\,_s C_R ^\lambda  $ is isomorphic to the direct sum of $\left| {L\left( \lambda  \right)} \right|$ copies of $C_R ^\lambda  $.

     For each $\lambda  \in \Lambda $ there is an $R$-bilinear map $ \left\langle \,\,,\,\,\right\rangle :\,C_R ^\lambda   \times \,_L C^\lambda \to R$ defined on basis elements by $\left\langle {C_t ^\lambda  ,\,_s C^\lambda  } \right\rangle  = r_{st} $, where $r_{st}  \in R$ is as given in lemma \ref {l2.1}.

\begin {definition} \label {d3.1}  The right $C_R ^\lambda  $ radical is \[ \rad \left( {C_R ^\lambda  } \right) = \left\{ {x \in C_R ^\lambda  :\left\langle {x,y} \right\rangle  = 0{\text{ for all }}y \in \,_L C^\lambda  } \right\} . \]
The left $\,_L C^\lambda  $ radical is
\[ \rad \left( {\,_L C^\lambda  } \right) = \left\{ {y \in \,_L C^\lambda  :\left\langle {x,y} \right\rangle  = 0{\text{ for all }}x \in C_R ^\lambda  } \right\} . \]
\end{definition}

     The radical $\rad \left( {C_R ^\lambda  } \right)$ is a right $A$-submodule of $C_R ^\lambda  $  and  $\rad \left( {\,_L C^\lambda  } \right)$ is a left $A$-submodule of $\,_L C^\lambda  $.

\begin {definition} \label {d3.2}  $D_R ^\lambda   = \frac{{C_R ^\lambda  }}
{{ \rad \left( {C_R ^\lambda  } \right)}}$,  $\,_L D^\lambda   = \frac{{\,_L C^\lambda  }}
{{ \rad \left( {\,_L C^\lambda  } \right)}}$  .
\end{definition}

     Then $D_R ^\lambda  $ is a right $A$-module and $\,_L D^\lambda  $ is a left $A$-module.  The following lemma follows at once from the definitions.

\begin {lemma} \label {l3.4}  The following conditions are equivalent: \\
(i) $D_R ^\lambda   = 0$; \   (ii) $\rad \left( {C_R ^\lambda  } \right) = C_R ^\lambda  $;   \       (iii) $\left\langle {x,y} \right\rangle  = 0{\text{ for all }}x \in C_R ^\lambda  \,,\,y \in \,_L C^\lambda  $; (iv) $\rad \left( {\,_L C^\lambda  } \right) = \,_L C^\lambda  $; and (v) $\,_L D^\lambda   = 0$.
\end{lemma}

\begin {definition} \label {d3.3}  $\Lambda _0  = \left\{ {\lambda  \in \Lambda :\left\langle {x,y} \right\rangle  \ne 0{\text{ for some }}x \in C_R ^\lambda  ,y \in \,_L C^\lambda  } \right\}$.
\end{definition}

     Evidently, $\lambda  \in \Lambda _0  \Leftrightarrow D_R ^\lambda   \ne 0 \Leftrightarrow \,_L D^\lambda   \ne 0$.

     When $R = k$ is a field, one can characterize the irreducible modules in a cell algebra in terms of the set $\Lambda _0 $. 

\begin {proposition} \label {p4.1}  Let $R = k$ be a field and take $\lambda  \in \Lambda _0 $.  Then
  \begin{enumerate} [\upshape(a)]    
	\item $D_R ^\lambda  $ is an irreducible right $A$-module.
	\item $\rad \left( {C_R ^\lambda  } \right)$ is the unique maximal right submodule in $C_R ^\lambda  $.
	\item $\,_L D^\lambda  $ is an irreducible left $A$-module.
	\item $\rad \left( {\,_L C^\lambda  } \right)$ is the unique maximal left submodule in $\,_L C^\lambda  $.
     \end{enumerate}
\end{proposition}

The modules $\left\{ {D_R ^\lambda  :\lambda  \in \Lambda _0 } \right\}$ are shown to be absolutely irreducible and pairwise inequivalent and similarly for $\left\{ {\,_L D^\lambda  :\lambda  \in \Lambda _0 } \right\}$. 

A major result of \cite{May3} is the following:

\begin {thm} \label {t4.1}  Assume $R = k$ is a field.  Then $\left\{ {D_R ^\mu  :\mu  \in \Lambda _0 } \right\}$ is a complete set of pairwise inequivalent irreducible right $A$-modules and $\left\{ {\,_L D^\mu  :\mu  \in \Lambda _0 } \right\}$ is a complete set of pairwise inequivalent irreducible left $A$-modules.
\end{thm}

In \cite{May3} it the following result is also obtained:

\begin{corollary} \label{c6.1}  If $\Lambda  = \Lambda _0 $, then $A$ is quasi-hereditary.
\end{corollary}

\section {Cell bases for certain monoid algebras}

      In this section we review (again omitting most of the proofs) the cell bases given in \cite{May3} for the monoid algebras $R\left[ {M} \right]$ corresponding to a class of monoids $M$ containing the full transformation semigroups $\mathcal{T} _r $ and the rook monoids $\Re _r $.  Some of the notation and results will be needed in the next section on generalized Schur algebras.
   
     Let $\bar r = \left\{ {1,2, \cdots ,r} \right\}$ and let $\bar \tau _r $ be the monoid of all maps \\ $\alpha :\bar r \cup \left\{ 0 \right\} \to \bar r \cup \left\{ 0 \right\}$ such that $\alpha \left( 0 \right) = 0$.  Note that $\bar \tau _r $ can be identified with the partial transformation semigroup $\mathcal{PT}_r$ of all ``partial maps'' of $\bar r$ to itself.  The full transformation semigroup $\mathcal{T} _r $ of all maps $\bar r \to \bar r$ can be identified with the submonoid of $\bar \tau _r $ consisting of maps with $\alpha ^{ - 1} \left( 0 \right) = 0$.  The rook monoid $\Re _r $ can be identified with the submonoid of $\bar \tau _r $ consisting of maps such that $\alpha ^{ - 1} \left( i \right)$ has at most one element for each $i \in \bar r$.  With these identifications, the symmetric group $\mathfrak{S}_r $ is the intersection $\mathcal{T} _r  \cap \Re _r $.

     Let $M$ be any monoid contained in $\bar \tau _r $ and containing $\mathfrak{S}_r $.  Let $R$ be a commutative domain with unit 1 and let $R\left[ M \right]$ be the monoid algebra over $R$.  We will describe a cell basis for $R\left[ M \right]$.

     For $\alpha  \in M$, the index of $\alpha $ is the number of nonzero elements in the image of $\alpha $, ${\text{index}}\left( \alpha  \right) = \left| {{\text{image}}\left( \alpha  \right) - \left\{ 0 \right\}} \right|$.  Let $I\left( M \right) \subseteq \bar r \cup \left\{ 0 \right\}$ be the set of indices of elements in $M$, that is, 
\[ I\left( M \right) = \left\{ {i:\exists \alpha  \in M{\text{ with index}}\left( \alpha  \right) = i} \right\}. \]
For $i \in \bar r$, let $\Lambda \left( i \right)$ be the set of all (integer) partitions of $i$.  Let $\Lambda \left( 0 \right)$ be a set with one element $\lambda _0 $.  Then define $\Lambda  = \mathop  \cup \nolimits_{i \in I(M)} \Lambda \left( i \right)$.  For $\lambda  \in \Lambda $ define the index $i\left( \lambda  \right)$ to be the integer such that $\lambda  \in \Lambda \left( {i\left( \lambda  \right)} \right)$.  Finally, define a partial order on $\Lambda $ by
\[ \lambda  \geqslant \mu  \Leftrightarrow i\left( \lambda  \right) < i\left( \mu  \right){\text{ or }}i\left( \lambda  \right) = i\left( \mu  \right){\text{ and }}\lambda \underset{\raise0.3em\hbox{$\smash{\scriptscriptstyle-}$}}{ \triangleright } \mu \]
where $\underset{\raise0.3em\hbox{$\smash{\scriptscriptstyle-}$}}{ \triangleright } $ is the usual dominance relation on partitions.

     To define the sets $L\left( \lambda  \right)$ and $R\left( \lambda  \right)$ we need some preliminaries.  First, for $i \in \bar r \cup \left\{ 0 \right\}$ let $C\left( {i,r} \right)$ be the collection of all $i$-sets of elements in $\bar r$, that is, $C(i,r) = \left\{ {C = \left\{ {c_1 ,c_2 , \ldots ,c_i } \right\}:1 \leqslant c_1  < c_2  <  \cdots  < c_i  \leqslant r} \right\}$.  ($C(0,r)$ contains one element, the empty set.)  For any $C \in C(i,r)$, define a map $\phi _C :\bar i \cup \left\{ 0 \right\} \to \bar r \cup \left\{ 0 \right\}$ by $\phi _C (j) = c_j {\text{ for }}j \in \bar i$, $\phi _C (0) = 0$.

     Next, choose any ordering of the $2^r $ subsets of $\bar r$ and label these subsets $d_j $ so that $d_1  < d_2  <  \cdots  < d_{2^r } $.  Let $D(i,r)$ be the collection of sets of $i$ nonempty, pairwise disjoint subsets of $\bar r$, that is, 
     \[D\left( {i,r} \right) = \left\{  \left\{ d_{a_1 } ,d_{a_2 } , \cdots ,d_{a_i }  \right\}:d_{a_j }  \ne \emptyset ,d_{a_j }  \cap d_{a_k }  = \emptyset \text{ and }d_{a_j }  < d_{a_k } \text{ for } j < k \right\} .\]  ($D(0,r)$ also contains one element, the empty set.)  For any $D \in D\left( {i,r} \right)$ define a map $\psi _D :\bar r \cup \left\{ 0 \right\} \to \bar i \cup \left\{ 0 \right\}$ by $\psi _D (x) = j{\text{ for }}x \in d_{a_j } $,  $\psi _D (x) = 0{\text{ when }}x \notin d_{a_j } {\text{ for any }}j$.

     Regard an element $\sigma $ in the symmetric group $\mathfrak{S}_i $ as a mapping $\sigma :\bar i \cup \{ 0\}  \to \bar i \cup \left\{ 0 \right\}$ such that $\sigma \left( 0 \right) = 0$. Then for any $\sigma  \in \mathfrak{S}_i \,,\,C \in C(i,r)\,,\,D \in D(i,r)$, define an element $\alpha  = \alpha \left( {\sigma ,C,D} \right) \in \bar \tau _r $ by $\alpha  = \phi _C  \circ \sigma  \circ \psi _D $.  Then $\alpha \left( {\sigma ,C,D} \right)$ has index $i$.  Note that $\alpha \left( x \right) = \left\{ {\begin{array}{*{20}c}
        {c_{\sigma (j)} {\text{ if }}x \in D_{\alpha _j } }  \\
        {0{\text{ if }}x \notin D_{\alpha _j } {\text{ for any }}j}  \\
     
      \end{array} } \right.$  .

\begin{lemma} \label{ll2.1}
   For any $\alpha  \in \bar \tau _r $ of index $i$, there exist unique $\sigma _\alpha   \in \mathfrak{S}_i \,,\,C_\alpha   \in C(i,r)\,,\,D_\alpha   \in D(i,r)$ such that $\alpha  = \phi _{C_\alpha  }  \circ \sigma _\alpha   \circ \psi _{D_\alpha  } $.
\end{lemma}

     Notice that $\mathfrak{S}_r $ acts on the left on $C\left( {i,r} \right)$:  for $C = \left\{ {c_1 ,c_2 , \cdots ,c_i } \right\} \in C\left( {i,r} \right)$ and $\sigma  \in \mathfrak{S}_r $, let $\sigma C = \left\{ {\sigma c_1 ,\sigma c_2 , \cdots ,\sigma c_i } \right\}$.

\begin{lemma} \label{ll2.2}
  Given $C,C' \in C\left( {i,r} \right){\text{ and }}\pi  \in \mathfrak{S}_i $, there exists a $\sigma  \in \mathfrak{S}_r $ such that $C' = \sigma C$ and $\sigma  \circ \phi _C  = \phi _{C'}  \circ \pi $.
\end{lemma}

     Given $C \in C\left( {i,r} \right)\,,\,D \in D\left( {i,r} \right)$, let $A\left( {C,D} \right)$ be the free $R$-submodule of $R[\bar \tau _r ]$ with basis $\left\{ {\alpha  \in \bar \tau _r :C_\alpha   = C,D_\alpha   = D} \right\}$.  Note that if $i = 0$, then $C\left( {0,r} \right) = D(0,r) = \left\{ \emptyset  \right\}$, a set with one element. $A\left( {\emptyset ,\emptyset } \right)$ is then one dimensional with basis $z$, where $z$ is the zero map such that $z\left( j \right) = 0$ for all $j \in \bar r \cup \left\{ 0 \right\}$. Evidently, as an $R$-module 
\[ R\left[ {\bar \tau _r } \right] = \mathop  \oplus \limits_{i \in \bar r \cup \left\{ 0 \right\}} \left( {\mathop  \oplus \limits_{C \in C\left( {i,r} \right),D \in D\left( {i,r} \right)} A\left( {C,D} \right)} \right) . \]

\begin{lemma} \label{ll2.3}
  Suppose that for $D \in D\left( {i,r} \right)$ there exists an $\alpha  \in M$ with $D_\alpha   = D$.  Then $A\left( {C,D} \right) \cap R\left[ M \right] = A\left( {C,D} \right)$ for every $C \in C\left( {i,r} \right)$.
\end{lemma}

     Define $D\left( {M,i,r} \right) = \left\{ {D \in D\left( {i,r} \right):\exists \alpha  \in M{\text{ with }}D_\alpha   = D} \right\}$.  Then as an $R$-module, $R\left[ M \right] = \mathop  \oplus \limits_{i \in \bar r \cup \left\{ 0 \right\}} \left( {\mathop  \oplus \limits_{C \in C\left( {i,r} \right),D \in D\left( {M,i,r} \right)} A\left( {C,D} \right)} \right)$  by lemma~\ref{ll2.3}.  So choosing a basis for each free $R$-module $A\left( {C,D} \right)\,,\,C \in C\left( {i,r} \right)\,,\,D \in D\left( {M,i,r} \right)$ will give a basis for $R\left[ M \right]$.

\begin{definition} \label{dd2.2}
  For $C \in C\left( {i,r} \right)\,,\,D \in D\left( {i,r} \right)$, $i > 0$, define a map of $R$-modules $H_{C,D} :R\left[ {\mathfrak{S}_i } \right] \to A\left( {C,D} \right)$ by $H_{C,D} \left( \sigma  \right) = \phi _C  \circ \sigma  \circ \psi _D $.
\end{definition}

By lemma \ref{ll2.1}, $H_{C,D} $ is well-defined and is a bijection between free $R$-modules.  So any basis for $R\left[ {\mathfrak{S}_i } \right]$ transfers to a basis for $A\left( {C,D} \right)$.  Let $B_i  = \left\{ {_s C_t ^\lambda  :\lambda  \in \Lambda (i)\,,\,s,t{\text{ standard }}\lambda {\text{ tableaux}}} \right\}$ be the standard Murphy cellular basis for the cellular algebra $R\left[ {\mathfrak{S}_i } \right]$ (See e.g. \cite{GL} or \cite{Mathas} ).  Then $\left\{ {H_{C,D} \left( {\,_s C_t ^\lambda  } \right):\,_s C_t ^\lambda   \in B_i } \right\}$ is a basis for $A\left( {C,D} \right)$.  

      We can now finally define our index sets $L\left( \lambda  \right)$ and $R\left( \lambda  \right)$.   Given $\lambda  \in \Lambda \left( i \right)\,,\,i \in I\left( M \right)$, $i > 0$, define 
\[L\left( \lambda  \right) = \left\{ {\left( {C,s} \right):C \in C\left( {i,r} \right),\,\,s{\text{ a standard }}\lambda {\text{ tableau}}} \right\} \]
 and 
\[ R\left( \lambda  \right) = \left\{ {\left( {D,t} \right):D \in D\left( {M,i,r} \right)\,,\,t{\text{ a standard }}\lambda {\text{ tableau}}} \right\}. \]
Then for any $\lambda  \in \Lambda \,,\,\left( {C,s} \right) \in L\left( \lambda  \right)\,,\,\left( {D,t} \right) \in R\left( \lambda  \right)$ define 
\[ {}_{\left( {C,s} \right)}C_{\left( {D,t} \right)}^\lambda   = H_{C,D} \left( {_s C_t ^\lambda  } \right) \in A\left( {C,D} \right) \subseteq R\left[ M \right]. \]
If $0 \in I\left( M \right)$, that is, if the zero map $z$ such that $z\left( j \right) = 0$ for all $j \in \bar r \cup \left\{ 0 \right\}$ is in $M$, we define $\Lambda \left( 0 \right)$ to have a single element $\lambda _0 $ and define $L\left( {\lambda _0 } \right) = R\left( {\lambda _0 } \right) = \left\{ \emptyset  \right\}$  each to be sets containing one element, $\emptyset $.  Then define ${}_\emptyset C_\emptyset ^{\lambda _0 }  = z$.   The set
\[ \left\{ {{}_{\left( {C,s} \right)}C_{\left( {D,t} \right)}^\lambda  :\lambda  \in \Lambda \,,\,\left( {C,s} \right) \in L\left( \lambda  \right)\,,\,\left( {D,t} \right) \in R\left( \lambda  \right)} \right\} \]
is a union of the bases for the various direct summands $A\left( {C,D} \right)$ and is therefore a basis for the free $R$-module $R\left[ M \right]$.   We will show that it is a cell-basis for $R\left[ M \right]$.
   
     Write $A_i $ for the cellular algebra $R\left[ {\mathfrak{S}_i } \right]$ and $\hat A_i ^\lambda  $ for the two sided ideal in $R\left[ {\mathfrak{S}_i } \right]$ spanned by $\left\{ {\,_s C_t ^\mu  :\mu  > \lambda } \right\}$.  The following observation will be useful.  Recall that $\hat A^\lambda  $ is the $R$-submodule of $R[M]$ spanned by $\left\{ {_{(C,s)} C^\mu  _{(D,t)} :\mu  > \lambda } \right\}$.

\begin{lemma} \label{ll2.4}
  For any $C \in C\left( {i,r} \right)\,,\,D \in D\left( {i,r} \right)$ and $\lambda  \in \Lambda \left( i \right)$, $H_{C,D} \left( {\hat A_i ^\lambda  } \right) \subseteq \hat A^\lambda  $.
\end{lemma}

\begin{lemma} \label{ll2.5}
  For $\alpha  \in M\,,\,C \in C\left( {i,r} \right)$, suppose $C' = \alpha \left( C \right) \in C\left( {i,r} \right)$.  Then there exists $\rho  \in \mathfrak{S}_i $ such that $\alpha  \circ \phi _C  = \phi _{C'}  \circ \rho $.
\end{lemma}

\begin{lemma} \label{ll2.6}
  For $\alpha  \in M\,,\,D = \left\{ {d_{a_j } :j \in \bar i} \right\} \in D\left( {i,r} \right)$, suppose that $\alpha ^{ - 1} \left( {d_{a_j } } \right) \ne \emptyset $  for all $j$, so that $D' = \left\{ {\alpha ^{ - 1} \left( {d_{a_j } } \right):j \in \bar i} \right\} \in D\left( {i,r} \right)$.   Then there exists $\rho  \in \mathfrak{S}_i $ such that $\psi _D  \circ \alpha  = \rho  \circ \psi _{D'} $.  Furthermore, if $D \in D\left( {M,i,r} \right)$, then $D' \in D\left( {M,i,r} \right)$.
\end{lemma}

\begin{proposition} \label{pp2.1}
  $C = \left\{ {{}_{\left( {C,s} \right)}C_{\left( {D,t} \right)}^\lambda  :\lambda  \in \Lambda \,,\,\left( {C,s} \right) \in L\left( \lambda  \right)\,,\,\left( {D,t} \right) \in R\left( \lambda  \right)} \right\}$ is a cell basis for $A = R\left[ M \right]$. \end{proposition}

      For $\lambda  \in \Lambda $, the right cell module  $C_R ^\lambda  $ is a right $A$-module and a free $R$-module with basis $\left\{ {C_{(D,t)} ^\lambda  :(D,t) \in R\left( \lambda  \right)} \right\}$, while the left cell module $\,_L C^\lambda  $ is a left $A$-module and a free $R$-module with basis $\left\{ {\,_{(C,s)} C^\lambda  :(C,s) \in L\left( \lambda  \right)} \right\}$.  The bracket for $A$ is an $R$-bilinear map $\left\langle { - , - } \right\rangle :C_R ^\lambda   \times \,_L C^\lambda   \to R$ defined on basis elements by $\left\langle {C_{(D,t)} ^\lambda  ,\,_{(C,s)} C^\lambda  } \right\rangle  = r_{(C,s),(D,t)}  \in R$  where ${}_{\left( {C',s'} \right)}C_{\left( {D,t} \right)}^\lambda   \cdot {}_{\left( {C,s} \right)}C_{\left( {D',t'} \right)}^\lambda   = r_{(C,s),(D,t)}  \cdot {}_{\left( {C',s'} \right)}C_{\left( {D',t'} \right)}^\lambda  \,\bmod \hat A^\lambda  $.

\begin{lemma} \label{ll3.1}
  Assume $C,C' \in C\left( {i,r} \right)\,,\,D,D' \in D\left( {i,r} \right)$ and $x,y \in R\left( {\mathfrak{S}_i } \right)$.  Then
   \begin{enumerate} [\upshape(a)]
      \item If $\rho  = \psi _D  \circ \phi _C :\bar i \to \bar i$ is not bijective, then $H_{C',D} \left( x \right) \cdot H_{C,D'} \left( y \right) \in J_{i - 1} $.
      \item If $\rho  = \psi _D  \circ \phi _C :\bar i \to \bar i$ is bijective, then for any $\pi  \in \mathfrak{S}_i $ ,	    			$H_{C',D} \left( x \right) \cdot H_{C,D'} \left( {\pi y} \right) = H_{C',D'} \left( {x\rho \pi y} \right)$.
   \end{enumerate}
\end{lemma}

\begin{lemma} \label{ll3.2}
  Let $i \in I\left( M \right)$, $\lambda  \in \Lambda \left( i \right)$, $(C,s) \in L\left( \lambda  \right)$, and $\left( {D,t} \right) \in R\left( \lambda  \right)$.  Then
\begin{enumerate} [\upshape(a)] 
\item  If  $\rho  = \psi _D  \circ \phi _C :\bar i \to \bar i$ is not bijective, $\left\langle {C_{\left( {D,t} \right)} ^\lambda  ,\,_{\left( {C,s} \right)} C^\lambda  } \right\rangle  = 0$.
\item If $\rho  = \psi _D  \circ \phi _C :\bar i \to \bar i$ is bijective, then for any $\pi  \in \mathfrak{S}_i $,
\[ \left\langle {C_{\left( {D,t} \right)} ^\lambda  ,\,\pi ' \cdot \,_{\left( {C,s} \right)} C^\lambda  } \right\rangle  = \left\langle {C_t ^\lambda  ,\rho \pi  \cdot \,_s C^\lambda  } \right\rangle _i . \] 
\end{enumerate}
  Here $\left\langle { - , - } \right\rangle _i $ is the bracket in the cellular algebra $R\left[ {\mathfrak{S}_i } \right]$ and $\pi ' = H_{C,D''} \left( \pi  \right)$ for any $D'' \in D\left( {M,i,r} \right)$ such that $\psi _{D''}  \circ \phi _C  = id:\bar i \to \bar i$.
\end{lemma}

Recall that the radical, ${\text{rad}}\left( {C_R ^\lambda  } \right)$, of a right cell module is the right $A$-module given by ${\text{rad}}\left( {C_R ^\lambda  } \right) = \left\{ {x \in C_R ^\lambda  :\left\langle {x,y} \right\rangle  = 0{\text{ for all }}y \in \,_L C^\lambda  } \right\}$.

\begin{proposition} \label{pp3.1}
  Let  $i \in I\left( M \right)$, $\lambda  \in \Lambda \left( i \right)$.   Then ${\text{rad}}\left( {C_R ^\lambda  } \right) = C_R ^\lambda  {\text{ in }}A \Leftrightarrow {\text{rad}}\left( {C^\lambda  } \right) = C^\lambda  {\text{ in }}A_i $.
\end{proposition}

\begin{proof}
  Assume first that ${\text{rad}}\left( {C^\lambda  } \right) = C^\lambda  {\text{ in }}A_i $, so $\left\langle {x,y} \right\rangle _i  = 0$ for all $x,y$.  To show ${\text{rad}}\left( {C_R ^\lambda  } \right) = C_R ^\lambda  $ it suffices to show that $\left\langle {C_{\left( {D,t} \right)}^\lambda  ,{}_{\left( {C,s} \right)}C^\lambda  } \right\rangle  = 0$ for any $(D,t) \in R(\lambda ),(C,s) \in L\left( \lambda  \right)$.  If   $\rho  = \psi _D  \circ \phi _C :\bar i \to \bar i$ is not bijective, lemma \ref{ll3.2}(a) gives $\left\langle {C_{\left( {D,t} \right)} ^\lambda  ,\,_{\left( {C,s} \right)} C^\lambda  } \right\rangle  = 0$.  If $\rho  = \psi _D  \circ \phi _C :\bar i \to \bar i$ is bijective, take $\pi $ in lemma \ref{ll3.2} to be the identity so that $\pi '{}_{\left( {C,s} \right)}C^\lambda   = {}_{\left( {C,s} \right)}C^\lambda  $.  Then by lemma \ref{ll3.2}(b), $\left\langle {C_{\left( {D,t} \right)} ^\lambda  ,_{\left( {C,s} \right)} C^\lambda  } \right\rangle  = \left\langle {C_{\left( {D,t} \right)} ^\lambda  ,\,\pi ' \cdot \,_{\left( {C,s} \right)} C^\lambda  } \right\rangle  = \left\langle {C_t ^\lambda  ,\rho  \cdot \,_s C^\lambda  } \right\rangle _i  = 0$.

     Now assume  ${\text{rad}}\left( {C_R ^\lambda  } \right) = C_R ^\lambda  {\text{ in }}A$, so $\left\langle {x,y} \right\rangle  = 0$ for any $x \in C_R ^\lambda  ,y \in \,_L C^\lambda  $.  To show ${\text{rad}}\left( {C^\lambda  } \right) = C^\lambda  {\text{ in }}A_i $ it suffices to show that  $\left\langle {C_t^\lambda  ,{}_sC^\lambda  } \right\rangle _i  = 0$ for any $t,s$.  Take any $D \in D\left( {M,i,r} \right)$ and choose $C \in C\left( {i,r} \right)$ such that $\rho  = \psi _D  \circ \phi _C $ is bijective.  Then apply lemma \ref{ll3.2}(b) with $\pi  = \rho ^{ - 1} $ to get  $\left\langle {C_t ^\lambda  ,\,_s C^\lambda  } \right\rangle _i  = \left\langle {C_t ^\lambda  ,\rho \pi  \cdot \,_s C^\lambda  } \right\rangle _i  = \left\langle {C_{\left( {D,t} \right)} ^\lambda  ,\,\pi ' \cdot \,_{\left( {C,s} \right)} C^\lambda  } \right\rangle  = 0$.
\end{proof}

     Note:  In the special case when $0 \in I\left( M \right)$, so $\Lambda \left( 0 \right) = \left\{ {\lambda _0 } \right\} \subseteq \Lambda $, the cell modules $C_R^{\lambda _0 } ,{}_LC^{\lambda _0 } $ are one dimensional with generators $C_\emptyset ^{\lambda _0 } ,{}_\emptyset C^{\lambda _0 } $ and $\left\langle {C_\emptyset ^\lambda  ,{}_\emptyset C^\lambda  } \right\rangle  = 1$ (since $z \cdot z = z$ where $z = {}_\emptyset C_\emptyset ^{\lambda _0 } :j \mapsto 0$ for all $j \in \bar i \cup \left\{ 0 \right\}$).  Then ${\text{rad}}\left( {C_R^{\lambda _0 } } \right) = 0$.

\section  {Generalized Schur algebras}

     For a monoid $M$ as in section 3 and a domain $R$, a ``generalized Schur algebra'',  $S\left( {M,R} \right)$, was defined in \cite{MA} and \cite{May}.  This algebra is isomorphic to $R \otimes A^\mathbb{Z} $ where $A^\mathbb{Z} $ is a certain  ``$\mathbb{Z}$-form'' .  As shown in \cite{May2}, there are actually two relevant $\mathbb{Z}$-forms, the left and right generalized Schur algebras $LGS\left( {M,{\mathbf{G}}} \right) = A_L ^\mathbb{Z} \,{\text{ and  }}RGS\left( {M,{\mathbf{G}}} \right) = A_R ^\mathbb{Z} $, corresponding to the monoid $M$ and the family of subgroups ${\mathbf{G}} = \left\{ {\mathfrak{S}_\mu  :\mu  \in \Lambda \left( {r,n} \right)} \right\}$ .  Here $\Lambda \left( {r,n} \right)$ is the set of all compositions of $r$ with $n$ parts, $\mathfrak{S}_\mu  $ is the ``Young subgroup'' corresponding to $\mu  \in \Lambda \left( {r,n} \right)$, and we will always assume $n \geqslant r $. We sketch the description of these two algebras; for details see \cite{May2}.

   For compositions $\mu ,\nu  \in \Lambda \left( {r,n} \right)$, let ${}^\mu A^\nu  $ be the $\mathbb{Z}$-submodule of $A = \mathbb{Z}\left[ M \right]$ which is invariant under the action of  $\mathfrak{S}_\mu  $ on the left and $\mathfrak{S}_\nu  $ on the right.  Let $\,_\mu  M_\nu  $ be the set of double cosets of the form $\mathfrak{S}_\mu  m\mathfrak{S}_\nu  $, $m \in M$.  For ${\mathbf{D}} \in \,_\mu  M_\nu  $, define $X\left( {\mathbf{D}} \right) = \sum\limits_{m \in {\mathbf{D}}} m  \in \mathbb{Z}\left[ M \right]$.   Then ${}^\mu A^\nu  $ is a free $\mathbb{Z}$-module with basis $\left\{ {X\left( {\mathbf{D}} \right):{\mathbf{D}} \in \,_\mu  M_\nu  } \right\}$.   Let $\bar A = \mathop  \oplus \limits_{\mu ,\nu  \in \Lambda \left( {r,n} \right)} \,^\mu  A^\nu  $, the direct sum of disjoint copies of submodules of $A$.  Then $\bar A$ is a free $\mathbb{Z}$ module with basis $\left\{ {b_{\mathbf{D}}  = X\left( {\mathbf{D}} \right):{\mathbf{D}} \in \,_\mu  M_\lambda  ,\mu ,\nu  \in \Lambda \left( {r,n} \right)} \right\}$.  Notice that if $D_1  \in \,_\mu  M_\nu  ,D_2  \in \,_\nu  M_\pi  $, then the product  $X\left( {D_1 } \right)X\left( {D_2 } \right)$ (defined in $A$) is invariant under multiplication by $\mathfrak{S}_\mu  $ on the left and by $\mathfrak{S}_\pi  $ on the right, i.e., $X\left( {D_1 } \right)X\left( {D_2 } \right) \in \,^\mu  A^\pi  $.  It is therefore a $\mathbb{Z}$-linear combination of $\left\{ {X\left( D \right):D \in \,_\mu  M_\pi  } \right\}$:  $X\left( {D_1 } \right)X\left( {D_2 } \right) = \sum\limits_{D \in \,_\mu  M_\pi  } {a\left( {D_1 ,D_2 ,D} \right)X\left( D \right)} $, with coefficients $a\left( {D_1 ,D_2 ,D} \right) \in \mathbb{Z}$.   Then an associative, bilinear product on $\bar A$ is defined on the basis elements $b_{D_i }  = X\left( {D_i } \right)$
corresponding to $D_1  \in \,_\mu  M_\nu  ,D_2  \in \,_\rho  M_\pi  $ by 
\[b_{D_1 } b_{D_2 }  = 
  \begin{cases}
   \sum\limits_{D \in \,_\mu  M_\pi  } {a\left( {D_1 ,D_2 ,D} \right)b_D } &\text{if $\nu  = \rho $}  \\
   \quad 0                  &\text{if $\nu  \ne \rho$ .}
   \end{cases}
\]
   $\bar A$ with this multiplication fails (in general) to have an identity.  To obtain the $\mathbb{Z}$-forms $A_L ^\mathbb{Z} \,{\text{ and  }}A_R ^\mathbb{Z} $ , which are $\mathbb{Z}$-algebras with identity, we define new ``left'' and ``right'' products, 
$ * _L $ and $ * _R $ on $\bar A$.

     For ${\mathbf{D}} \in \,_\mu  M_\nu  $, let $n_L \left( {\mathbf{D}} \right)$ be the number of elements in any left $\mathfrak{S}_\mu  $-coset $C \subseteq {\mathbf{D}}$.  Then the product $ * _L $ is defined on the basis elements $b_{{\mathbf{D}}_i }  = X\left( {{\mathbf{D}}_i } \right)$ corresponding to ${\mathbf{D}}_1  \in \,_\mu  M_\nu  ,{\mathbf{D}}_2  \in \,_\rho  M_\pi  $ by 
\[b_{{\mathbf{D}}_1 }  * _L b_{{\mathbf{D}}_2 }  =
     \begin{cases}
       \sum\limits_{{\mathbf{D}} \in \,_\mu  M_\pi  } {\frac{{n_L \left( {\mathbf{D}} \right)}}
   {{n_L \left( {{\mathbf{D}}_1 } \right)n_L \left( {{\mathbf{D}}_2 } \right)}} \, a\left( {{\mathbf{D}}_1 ,{\mathbf{D}}_2 ,{\mathbf{D}}} \right)b_{\mathbf{D}} } &\text{if $\nu  = \rho $}  \\
   \quad 0   &\text{if $\nu  \ne \rho$ .}
   \end{cases}
   \]
      Similarly, for ${\mathbf{D}} \in \,_\mu  M_\nu  $, let $n_R \left( {\mathbf{D}} \right)$ be the number of elements in any right $\mathfrak{S}_\nu  $-coset $C \subseteq {\mathbf{D}}$.  Then the product $ * _R $ is defined on the basis elements $b_{{\mathbf{D}}_i }  = X\left( {{\mathbf{D}}_i } \right)$ corresponding to ${\mathbf{D}}_1  \in \,_\mu  M_\nu  ,{\mathbf{D}}_2  \in \,_\rho  M_\pi  $ by 
    \[b_{{\mathbf{D}}_1 }  * _R b_{{\mathbf{D}}_2 }  = 
       \begin{cases}
       \sum\limits_{{\mathbf{D}} \in \,_\mu  M_\pi  } {\frac{{n_R \left( {\mathbf{D}} \right)}}
       {{n_R \left( {{\mathbf{D}}_1 } \right)n_R \left( {{\mathbf{D}}_2 } \right)}} \, a\left( {{\mathbf{D}}_1 ,{\mathbf{D}}_2 ,{\mathbf{D}}} \right)b_{\mathbf{D}} } &\text{if $\nu  = \rho$ }  \\
       \quad 0     &\text{if $\nu  \ne \rho$ .}
       \end{cases}
   \]    
   As shown in \cite{May2}, the structure constants for these multiplications are in fact in $\mathbb{Z}$ and the resulting $\mathbb{Z}$-algebras, $A_L ^\mathbb{Z} \,{\text{ and  }}A_R ^\mathbb{Z} $ , have identities.  We will show that both of these algebras, and the associated generalized Schur algebras $R \otimes A_L ^\mathbb{Z} \,{\text{ and  }}R \otimes A_R ^\mathbb{Z} $ for any domain $R$, have cell bases and are cell algebras.     

      We use the notation and definitions of section 3.  For a composition $\mu  \in \Lambda \left( {r,n} \right)$, $\mathfrak{S}_\mu  $
 acts on the left on $C\left( {i,r} \right)$ :  if $C = \left\{ {c_1 ,c_2 , \cdots ,c_i } \right\} \in C\left( {i,r} \right)$ and  $\sigma  \in \mathfrak{S}_\mu  $  define $\sigma C = \left\{ {\sigma c_1 , \cdots ,\sigma c_i } \right\}$.  Then write $O\left( {\mu ,C} \right) = \left\{ {\rho C:\rho  \in \mathfrak{S}_\mu  } \right\}$ for the orbit of $C \in C\left( {i,r} \right)$ under $\mathfrak{S}_\mu  $.  Similarly, for a composition $\nu  \in \Lambda \left( {r,n} \right)$, $\mathfrak{S}_\nu  $
 acts on the right on $D\left( {M,i,r} \right)$:  for $D = \left\{ {d_1 ,d_2 , \cdots ,d_i } \right\} \in D\left( {M,i,r} \right)$ and $\pi  \in \mathfrak{S}_\nu  $, $D\pi  = \left\{ {\pi ^{ - 1} \left[ {d_1 } \right], \cdots ,\pi ^{ - 1} [d_i ]} \right\}$.  Write $O\left( {\nu ,D} \right) = \left\{ {D\pi :\pi  \in \mathfrak{S}_\nu  } \right\}$ for the orbit of $D \in D\left( {M,i,r} \right)$ under $\mathfrak{S}_\nu  $.  Any double coset ${\mathbf{D}} = \mathfrak{S}_\mu  \alpha \mathfrak{S}_\nu   \in \,_\mu  M_\nu  $ has a well defined index $i$ since any $\beta  \in {\mathbf{D}}$ has the same index as $\alpha $.  We also have $O\left( {\mu ,C_\beta  } \right) = O\left( {\mu ,C_\alpha  } \right)$ for any $\beta  \in {\mathbf{D}}$, so ${\mathbf{D}}$ has a well defined orbit $O\left( {\mu ,{\mathbf{D}}} \right) = O\left( {\mu ,C_\alpha  } \right)$
  for the action of $\mathfrak{S}_\mu  $ on $C\left( {i,r} \right)$.   Similarly, ${\mathbf{D}}$ has a well defined orbit $O\left( {\nu ,{\mathbf{D}}} \right) = O\left( {\nu ,D_\alpha  } \right)$  for the action of $\mathfrak{S}_\nu  $ on $D\left( {M,i,r} \right)$ .  Let ${\mathbf{O}}_{\mu ,C\left( {i,r} \right)} $ be the set of orbits for the action of $\mathfrak{S}_\mu  $ on $C\left( {i,r} \right)$ so $C\left( {i,r} \right) = \bigcup\limits_{O \in {\mathbf{O}}_{\mu ,C\left( {i,r} \right)} } O $.  Similarly, let ${\mathbf{O}}_{\nu ,D(M,i,r)} $ be the set of orbits for the action of $\mathfrak{S}_\nu  $ on $D\left( {M,i,r} \right)$ so $D\left( {M,i,r} \right) = \bigcup\limits_{O \in {\mathbf{O}}_{\nu ,D\left( {M,i,r} \right)} } O $.   Define the set of orbit pairs ${\mathbf{O}}\left( {\mu ,\nu ,M,i} \right) = {\mathbf{O}}_{\mu ,C\left( {i,r} \right)}  \times {\mathbf{O}}_{\nu ,D\left( {M,i,r} \right)}  = \left\{ {\left( {O_\mu  ,O_\nu  } \right):O_\mu   \in {\mathbf{O}}_{\mu ,C\left( {i,r} \right)} ,O_\nu   \in {\mathbf{O}}_{\nu ,D\left( {M,i,r} \right)} } \right\}$.  Finally, for $\left( {O_\mu  ,O_\nu  } \right) \in {\mathbf{O}}\left( {\mu ,\nu ,M,i} \right)$ define a collection of double cosets 
  \[M\left( {O_\mu  ,O_\nu  } \right) = \left\{ {{\mathbf{D}} \in \,_\mu  M_\nu  :{\text{index}}({\mathbf{D}}) = i,O\left( {\mu ,{\mathbf{D}}} \right) = O_\mu  ,O\left( {\nu ,{\mathbf{D}}} \right) = O_\nu  } \right\}.\]
    Then for orbits $O_\mu   \in {\mathbf{O}}_{\mu ,C\left( {i,r} \right)} \,,\,O_\nu   \in {\mathbf{O}}_{\nu ,D\left( {M,i,r} \right)} $ define ${}^{O_\mu  }A^{O_\nu  } $ to be the free $\mathbb{Z}$-module with basis $\left\{ {X\left( {\mathbf{D}} \right):{\mathbf{D}} \in M\left( {O_\mu  ,O_\nu  } \right)} \right\}$.  Evidently 
    \[{}^\mu A^\nu   = \mathop  \oplus \limits_{i \in I\left( M \right)} \mathop  \oplus \limits_{(O_\mu  ,O_\nu  ) \in {\mathbf{O}}\left( {\mu ,\nu ,M,i} \right)} {}^{O_\mu  }A^{O_\nu  } .\]
      We will obtain a cell basis for $\bar A = \mathop  \oplus \limits_{\mu ,\nu  \in \Lambda \left( {r,n} \right)} \,^\mu  A^\nu  $ by taking the union of bases for the individual submodules ${}^{O_\mu  }A^{O_\nu  } $.       

      For $\mu  \in \Lambda \left( {r,n} \right)$ and $C \in C\left( {i,r} \right)$ let $\mu \left( C \right) \in \Lambda \left( {i,n} \right)$ be the composition of $i$ obtained as follows:   if $b_j ^\mu  ,\,j \in \bar n,$ is the $j$th block of $\mu $, then the $j$th block of $\mu \left( C \right)$ is given by  $b_j ^{\mu \left( C \right)}  = \phi _C ^{ - 1} \left( {b_j ^\mu  } \right),\,j \in \bar n$ .  Notice that $\mu \left( C \right)_j  = \left| {b_j ^\mu   \cap C} \right|$,  the number of the $i$ elements in $C$ which lie in the $j$th block of $\mu $.  Since elements of $\mathfrak{S}_\mu  $ preserve the blocks of $\mu $  (by definition), the composition $\mu \left( C \right)$ depends only on the orbit $O\left( {\mu ,C} \right)$, that is, $\mu \left( {\rho C} \right) = \mu \left( C \right)$ for any $\rho  \in \mathfrak{S}_\mu  $.  Let $\mathfrak{S}_{\mu \left( C \right)}  \subseteq \mathfrak{S}_i $ be the corresponding Young subgroup.

\begin{lemma}\label{l10.1}
  Given $C \in C\left( {i,r} \right)$ and $\rho  \in \mathfrak{S}_\mu  $ , let $\rho C \in O\left( {\mu ,C} \right)$ be the image of $C$ under $\rho $.  Then there exists a unique $\rho _C  \in \mathfrak{S}_{\mu \left( C \right)} $ such that $\rho  \cdot \phi _C  = \phi _{\rho C}  \cdot \rho _C $.  Conversely, given any $\rho _C  \in \mathfrak{S}_{\mu \left( C \right)} $ and any $\rho C \in O\left( {\mu ,C} \right)$, there exists a $\rho ' \in \mathfrak{S}_\mu  $ such that $\rho ' \cdot \phi _C  = \phi _{\rho C}  \cdot \rho _C $ and $\rho 'C = \rho C \in O\left( {\mu ,C} \right)$.
\end{lemma}

\begin{proof}
  $\rho  \cdot \phi _C $ and $\phi _{\rho C} $ both map $\bar i$ one to one onto $\rho C$.  So define $\rho _C  \in \mathfrak{S}_i $ by letting $\rho _C \left( k \right)$ be the unique element in $\phi _{\rho C} ^{ - 1} \left[ {\rho  \cdot \phi _C \left( k \right)} \right]$.  Then $\rho _C $ is the unique element in $\mathfrak{S}_i $ such that  $\rho  \cdot \phi _C  = \phi _{\rho C}  \cdot \rho _C $ , and we need only prove that $\rho _C  \in \mathfrak{S}_{\mu \left( C \right)} $ .  So suppose $j$ is in the $k$th $\mathfrak{S}_{\mu \left( C \right)} $  block $b_k ^{\mu \left( C \right)} $.  We must show that $\rho _C \left( j \right) \in b_k ^{\mu \left( C \right)} $.  But  $j \in b_k ^{\mu \left( C \right)}  \Rightarrow \phi _C \left( j \right) \in b_k ^\mu   \Rightarrow \rho  \cdot \phi _C \left( j \right) \in b_k ^\mu  $ (since $\rho  \in \mathfrak{S}_\mu  $).  Then $\phi _{\rho C}  \cdot \rho _C \left( j \right) \in b_k ^\mu  $ which implies $\rho _C \left( j \right) \in b_k ^{\mu \left( {\rho C} \right)} $.  But $\mu \left( {\rho C} \right) = \mu \left( C \right)$, so $\rho _C \left( j \right) \in b_k ^{\mu \left( C \right)} $ as desired.

     Now take any $\rho _C  \in \mathfrak{S}_{\mu \left( C \right)} $ and any $\rho C \in O\left( {\mu ,C} \right)$.  $\phi _C $ and $\phi _{\rho C}  \cdot \rho _C $ both take the $\mu \left( C \right)_j $ elements in the $j$th block of $\mu \left( C \right)$ one to one onto  $\mu \left( C \right)_j $ elements in the $j$th block of $\mu $.  So there exists $\rho ' \in \mathfrak{S}_\mu  $ such that $\rho ' \cdot \phi _C  = \phi _{\rho C}  \cdot \rho _C $.  Evidently $\rho 'C = {\text{image}}\left( {\rho ' \cdot \phi _C } \right) = {\text{image}}\left( {\phi _{\rho C}  \cdot \rho _C } \right) = \rho C$. 
\end{proof}

     Next, let $P_r $ be the set of all $2^r $ subsets of $\bar r$.   For a composition $\nu  \in \Lambda \left( {r,n} \right)$, $\mathfrak{S}_\nu  $ acts on the right on $P_r $:  if $S \in P_r $ and $\pi  \in \mathfrak{S}_\nu  $, then $S\pi  = \pi ^{ - 1} \left( S \right) = \left\{ {\pi ^{ - 1} \left[ s \right]:s \in S} \right\} \in P_r $ .  Choose a total order on the orbits of  $\mathfrak{S}_\nu  $ acting on $P_r $ and then label the orbits $O_i $ so that $O_1  < O_2  <  \cdots  < O_{N_\nu  } $ where $N_\nu  $ is the number of orbits.  Then choose a total order of all the subsets in $P_r $ which is compatible with the ordering of the $\mathfrak{S}_\nu  $ orbits:  $S_{a_i }  \in O_{a_i } $ and $O_{a_1 }  < O_{a_2 }  \Rightarrow S_{a_1 }  < S_{a_2 } $.  We label the subsets $S_i $ so that $S_1  < S_2  <  \cdots  < S_{2^r } $. 

     Now consider an element $D \in D\left( {M,i,r} \right)$.  $D$ consists of $i$ sets in $P_r $, so $D = \left\{ {S_{a_1 }  < S_{a_2 }  <  \cdots  < S_{a_i } } \right\}$.  Define a composition $ \nu \left( D \right) \in \Lambda \left( {i,N_\nu  } \right) $ by $\nu \left( D \right)_j  = \left| {D \cap O_j } \right|\,,\,j \in 1,2, \cdots ,N_\nu  $.   Then $S_{a_k } ,S_{a_l } $  are in the same orbit $O_j $ if and only if $k,l$ are in the same block of the composition $\nu \left( D \right)$, $k,l \in b_j ^{\nu \left( D \right)} $.   $\mathfrak{S}_\nu  $ acts on the right on $D\left( {M,i,r} \right)$:  For $\pi  \in \mathfrak{S}_\nu  $,$D\pi  = \left\{ {S_{a_j } \pi :j \in \bar i} \right\} = \left\{ {\pi ^{ - 1} \left[ {S_{a_j } } \right]:j \in \bar i} \right\} \in D\left( {M,i,r} \right)$.  Since by definition $\mathfrak{S}_\nu  $ preserves orbits, $\left| {D\pi  \cap O_j } \right| = \left| {D \cap O_j } \right|$ for any $\pi  \in \mathfrak{S}_\nu  $,  So the composition $\nu \left( {D\pi } \right) = \nu \left( D \right)$
 depends only on the orbit $O\left( {\nu ,D} \right)$ of $D$ under the action of $\mathfrak{S}_\nu  $.  Let $\mathfrak{S}_{\nu \left( D \right)}  \subseteq \mathfrak{S}_i $ be the corresponding Young subgroup.  

\begin{lemma} \label{l10.2}
  Given $D \in D\left( {M,i,r} \right)$ and $\pi  \in \mathfrak{S}_\nu $, let $D\pi  \in O\left( {\nu ,D} \right)$ be the image of $D$ under $\pi $.  Then there exists a unique $\pi _D  \in \mathfrak{S}_{\nu \left( D \right)} $ such that $\psi _D  \cdot \pi  = \pi _D  \cdot \psi _{D\pi } $.  Conversely, given any $\pi _D  \in \mathfrak{S}_{\nu \left( D \right)} $ and any $D\pi  \in O\left( {\nu ,D} \right)$, there exists a $\pi ' \in \mathfrak{S}_\nu  $ such that $\psi _D  \cdot \pi ' = \pi _D  \cdot \psi _{D\pi } $ and $D\pi ' = D\pi  \in O\left( {\nu ,D} \right)$.
\end{lemma}

\begin{proof}
  Let $D = \left\{ {S_{a_1 }  < S_{a_2 }  <  \cdots  < S_{a_i } } \right\}$, so $D\pi  = \left\{ {\pi ^{ - 1} \left( {S_{a_j } } \right):j \in \bar i} \right\}$.  Arrange the sets in $D\pi $ in order and define $k(j) \in \bar i$ such that $\pi ^{ - 1} \left( {S_{a_j } } \right)$ is the $k(j)$th set in the sequence.  Then $\psi _D  \cdot \pi $   maps elements in $\pi ^{ - 1} \left( {S_{a_j } } \right)$ to $j$, while $\psi _{D\pi } $ maps elements in $\pi ^{ - 1} \left( {S_{a_j } } \right)$ to $k(j)$.  Define  $\pi _D  \in \mathfrak{S}_i $ by $\pi _D \left( {k(j)} \right) = j\,,\,j \in \bar i$.  Then $\pi _D $ is the unique element in $\mathfrak{S}_i $ such that $\psi _D  \cdot \pi  = \pi _D  \cdot \psi _{D\pi } $, and it remains to show that $\pi _D  \in \mathfrak{S}_{\nu \left( D \right)} $.  For this, we must show that $j$ and $k(j)$ are always in the same block of the composition $\nu \left( D \right)$.  But $D\pi $ and $D$ contain the same number of subsets in each orbit of $\mathfrak{S}_\nu  $.  So if $S_{a_j } $ and  $\pi ^{ - 1} \left( {S_{a_j } } \right)$ are both in the $l$th orbit of $\mathfrak{S}_\nu  $  , then their indices $j$ and $k(j)$ are both in the $l$th block of $\nu \left( D \right)$~.

     Now take any $\pi _D  \in \mathfrak{S}_{\nu \left( D \right)} $ and any $D\pi  \in O\left( {\nu ,D} \right)$.  Recall that $\bar r \supseteq \bigcup\limits_{k \in \bar i} {\pi ^{ - 1} \left( {S_{a_k } } \right)} $
and that the sets $\pi ^{ - 1} \left( {S_{a_k } } \right)$ are pairwise disjoint, so we can define a $\pi ' \in \mathfrak{S}_r $ by defining $\left. {\pi '} \right|\pi ^{ - 1} \left( {S_{a_k } } \right)$ for each $k$.  Suppose $k$ is in the $j$th block of $\nu \left( D \right)$.  Then both $(\psi _D )^{ - 1} (k) = S_{a_k } $ and $\left( {\pi _D  \cdot \psi _{D\pi } } \right)^{ - 1} (k) = \pi ^{ - 1} \left( {S_{a_{m\left( k \right)} } } \right)$ (for some $m(k)$)  are in the $j$th orbit of  $\mathfrak{S}_\nu  $ .  Define $\pi ' \in \mathfrak{S}_r $ by $\left. {\pi '} \right|\left( {\pi ^{ - 1} \left( {S_{a_{m(k)} } } \right)} \right) = \left. {\sigma _k } \right|\left( {\pi ^{ - 1} \left( {S_{a_{m(k)} } } \right)} \right)$ where $\sigma _k  \in \mathfrak{S}_\nu  $ maps $\pi ^{ - 1} \left( {S_{a_{m(k)} } } \right)$ one to one onto $S_{a_k } $.  Then $\left( {\psi _D  \cdot \pi '} \right)^{ - 1} \left( k \right) = \left( {\pi '} \right)^{ - 1} \left( {\psi _D ^{ - 1} (k)} \right) = \left( {\pi '} \right)^{ - 1} \left( {S_{a_k } } \right) = \pi ^{ - 1} \left( {S_{a_{m(k)} } } \right) = \left( {\pi _D  \cdot \psi _{D\pi } } \right)^{ - 1} (k)$ for all $k$.  So  $\psi _D  \cdot \pi ' = \pi _D  \cdot \psi _{D\pi } $  and it remains to show that $\pi ' \in \mathfrak{S}_\nu  $.  It suffices to show that if $l$ is in the $j$th block of $\nu $ then $\pi '\left( l \right)$ is also in the $j$th block.  But any $l$ is in a unique $\pi ^{ - 1} \left( {S_{a_k } } \right)$, so $\pi '\left( l \right) = \sigma _k \left( l \right)$ is in fact in the $j$th block since $\sigma _k  \in \mathfrak{S}_\nu  $. 
\end{proof}

     Consider compositions $\mu ,\nu  \in \Lambda \left( {r,n} \right)$ and an element $\alpha  \in M$ of index $i$.  There exist unique $\sigma _\alpha   \in \mathfrak{S}_i \,,\,C_\alpha   \in C(i,r)\,,\,D_\alpha   \in D(M,i,r)$ such that $\alpha  = \phi _{C_\alpha  }  \circ \sigma _\alpha   \circ \psi _{D_\alpha  } $.  For $C = \rho C_\alpha   \in O\left( {\mu ,C_\alpha  } \right)$, define  \[\phi _C  \cdot \mathfrak{S}_{\mu \left( C \right)}  \cdot \sigma _\alpha   \circ \psi _{D_\alpha  }  = \left\{ {\phi _C  \cdot \kappa  \cdot \sigma _\alpha   \circ \psi _{D_\alpha  } :\kappa  \in \mathfrak{S}_{\mu \left( C \right)} } \right\} \subseteq M.\]  Similarly, for $D = D_\alpha  \pi  \in O\left( {\nu ,D_\alpha  } \right)$, define \[ \phi _{C_\alpha  }  \circ \sigma _\alpha   \cdot \mathfrak{S}_{\nu \left( D \right)}  \cdot \psi _D  = \left\{ {\phi _{C_\alpha  }  \circ \sigma _\alpha   \cdot \gamma  \cdot \psi _D :\gamma  \in \mathfrak{S}_{\nu \left( D \right)} } \right\}.\]  Finally, for $C \in O\left( {\mu ,C_\alpha  } \right)\,,\,D \in O\left( {\nu ,D_\alpha  } \right)$, define \[\phi _C  \cdot \mathfrak{S}_{\mu \left( C \right)}  \cdot \sigma _\alpha   \cdot \mathfrak{S}_{\nu \left( D \right)}  \cdot \psi _D  = \left\{ {\phi _C  \cdot \kappa  \cdot \sigma _\alpha   \cdot \gamma  \cdot \psi _D :\kappa  \in \mathfrak{S}_{\mu \left( C \right)} ,\gamma  \in \mathfrak{S}_{\nu \left( D \right)} } \right\} .\]

\begin{proposition} \label{p10.1}
   For compositions $\mu ,\nu  \in \Lambda \left( {r,n} \right)$ and an element $\alpha  = \phi _{C_\alpha  }  \circ \sigma _\alpha   \circ \psi _{D_\alpha  }  \in M$ of index $i$,

\begin{enumerate} [\upshape(a)]
    \item  For $C_1 ,C_2 ,C \in O\left( {\mu ,C_\alpha  } \right)$, $D_1 ,D_2 ,D \in O\left( {\nu ,D_\alpha  } \right)$, \[\left( {\phi _{C_1 }  \cdot \mathfrak{S}_{\mu \left( C \right)}  \cdot \sigma _\alpha   \cdot \mathfrak{S}_{\nu \left( D \right)}  \cdot \psi _{D_1 } } \right) \cap \left( {\phi _{C_2 }  \cdot \mathfrak{S}_{\mu \left( C \right)}  \cdot \sigma _\alpha   \cdot \mathfrak{S}_{\nu \left( D \right)}  \cdot \psi _{D_1 } } \right) = \emptyset \] unless $C_1  = C_2 $ and $D_1  = D_2 $.  The double coset $\mathfrak{S}_\mu  \alpha \mathfrak{S}_\nu  $ is a disjoint union 
    \[\mathfrak{S}_\mu  \alpha \mathfrak{S}_\nu   = \bigcup\limits_{C \in O\left( {\mu ,C_\alpha  } \right),D \in O\left( {\nu ,D_\alpha  } \right)} {\phi _C  \cdot \mathfrak{S}_{\mu \left( C \right)}  \cdot \sigma _\alpha   \cdot \mathfrak{S}_{\nu \left( D \right)}  \cdot \psi _D }  .\]
     \item  For $C_1 ,C_2  \in O\left( {\mu ,C_\alpha  } \right)$, \[\left( {\phi _{C_1 }  \cdot \mathfrak{S}_{\mu \left( C \right)}  \cdot \sigma _\alpha   \circ \psi _{D_\alpha  } } \right) \cap \left( {\phi _{C_2 }  \cdot \mathfrak{S}_{\mu \left( C \right)}  \cdot \sigma _\alpha   \circ \psi _{D_\alpha  } } \right) = \emptyset \] unless $C_1  = C_2 $.  \[\mathfrak{S}_\mu  \alpha  = \bigcup\limits_{C \in O\left( {\mu ,C_\alpha  } \right)} {\phi _C  \cdot \mathfrak{S}_{\mu \left( C \right)}  \cdot \sigma _\alpha   \circ \psi _{D_\alpha  } }  ,\] a disjoint union.
      \item  For  $D_1 ,D_2  \in O\left( {\nu ,D_\alpha  } \right)$, \[\left( {\phi _{C_\alpha  }  \cdot \sigma _\alpha   \cdot \mathfrak{S}_{\nu \left( D \right)}  \cdot \psi _{D_1 } } \right) \cap \left( {\phi _{C_\alpha  }  \cdot \sigma _\alpha   \cdot \mathfrak{S}_{\nu \left( D \right)}  \cdot \psi _{D_2 } } \right) = \emptyset \] unless $D_1  = D_2 $.  \[\alpha \mathfrak{S}_\nu   = \bigcup\limits_{D \in O\left( {\nu ,D_\alpha  } \right)} {\phi _{C_\alpha  }  \cdot \sigma _\alpha   \cdot \mathfrak{S}_{\nu \left( D \right)}  \circ \psi _D } ,\] a disjoint union.
\end{enumerate}
\end{proposition}

\begin{proof}
  For part (a), if \[\beta  \in (\phi _{C_1 }  \cdot \mathfrak{S}_{\mu \left( C \right)}  \cdot \sigma _\alpha   \cdot \mathfrak{S}_{\nu \left( D \right)}  \cdot \psi _{D_1 } ) \cap (\phi _{C_2 }  \cdot \mathfrak{S}_{\mu \left( C \right)}  \cdot \sigma _\alpha   \cdot \mathfrak{S}_{\nu \left( D \right)}  \cdot \psi _{D_1 } )\] then $C_1  = C_\beta   = C_2 $ and $D_1  = D_\beta   = D_2 $ , so \[(\phi _{C_1 }  \cdot \mathfrak{S}_{\mu \left( C \right)}  \cdot \sigma _\alpha   \cdot \mathfrak{S}_{\nu \left( D \right)}  \cdot \psi _{D_1 } ) \cap (\phi _{C_2 }  \cdot \mathfrak{S}_{\mu \left( C \right)}  \cdot \sigma _\alpha   \cdot \mathfrak{S}_{\nu \left( D \right)}  \cdot \psi _{D_1 } ) = \emptyset \] unless $C_1  = C_2 $ and $D_1  = D_2 $.  

If $\beta  \in \mathfrak{S}_\mu  \alpha \mathfrak{S}_\nu  $, then  $\beta  = \rho  \circ \phi _{C_\alpha  }  \circ \sigma _\alpha   \circ \psi _{D_\alpha  }  \circ \pi $ for some $\rho  \in \mathfrak{S}_\mu  \,,\,\pi  \in \mathfrak{S}_\nu  .$  By lemmas \ref{l10.1} and \ref{l10.2}, there exist $\rho _C  \in \mathfrak{S}_{\mu \left( C \right)} \,,\,\pi _D  \in \mathfrak{S}_{\nu \left( D \right)} $ such that  
\[\beta  = \phi _{\rho C_\alpha  }  \circ \rho _C  \circ \sigma _\alpha   \circ \pi _D  \circ \psi _{D_\alpha  \pi }  \in \phi _{\rho C_\alpha  }  \cdot \mathfrak{S}_{\mu \left( C \right)}  \circ \sigma _\alpha   \circ \mathfrak{S}_{\nu \left( D \right)}  \cdot \psi _{D_\alpha  \pi } \,\,.\]  So  $\mathfrak{S}_\mu  \alpha \mathfrak{S}_\nu   \subseteq \bigcup\limits_{C \in O\left( {\mu ,C_\alpha  } \right),D \in O\left( {\nu ,D_\alpha  } \right)} {\phi _C  \cdot \mathfrak{S}_{\mu \left( C \right)}  \cdot \sigma _\alpha   \cdot \mathfrak{S}_{\nu \left( D \right)}  \cdot \psi _D } $ .  On the other hand, if 
\[\beta  = \phi _{\rho C_\alpha  }  \circ \rho _C  \circ \sigma _\alpha   \circ \pi _D  \circ \psi _{D_\alpha  \pi }  \in \phi _{\rho C_\alpha  }  \cdot \mathfrak{S}_{\mu \left( C \right)}  \circ \sigma _\alpha   \circ \mathfrak{S}_{\nu \left( D \right)}  \cdot \psi _{D_\alpha  \pi } \, , \] then by lemmas \ref{l10.1} and \ref{l10.2} there exist $\rho ' \in \mathfrak{S}_\mu  \,,\,\pi ' \in \mathfrak{S}_\nu  $ such that   
\[\beta  = \rho ' \cdot \phi _{C_\alpha  }  \circ \sigma _\alpha   \circ \psi _{D_\alpha  }  \cdot \pi ' = \rho '\alpha \pi ' \in \mathfrak{S}_\mu  \alpha \mathfrak{S}_\nu  .\]  So $\bigcup\limits_{C \in O\left( {\mu ,C_\alpha  } \right),D \in O\left( {\nu ,D_\alpha  } \right)} {\phi _C  \cdot \mathfrak{S}_{\mu \left( C \right)}  \cdot \sigma _\alpha   \cdot \mathfrak{S}_{\nu \left( D \right)}  \cdot \psi _D }  \subseteq \mathfrak{S}_\mu  \alpha \mathfrak{S}_\nu  $ , completing the proof of part (a).
 
(b) follows from (a) by taking $D_1  = D_2  = D_\alpha  $ and $\nu $ to be the composition $\nu _i  = 1,\,\forall i$, so $\mathfrak{S}_\nu   = \mathfrak{S}_{\nu (D)}  = \left\{ 1 \right\}$.  Similarly, (c) follows from (a) by taking $C_1 = C_2  = C_\alpha  $ and $\mu $ to be the composition
$\mu _i~=~1,\,\forall i$, so $\mathfrak{S}_\mu   = \mathfrak{S}_{\mu (C)}  = \left\{ 1 \right\}$.  
\end{proof}

     Any double coset ${\mathbf{D}} = \mathfrak{S}_\mu  \alpha \mathfrak{S}_\nu   \in \,_\mu  M_\nu  $ has a well defined index $i$ since any $\beta  \in {\mathbf{D}}$ has the same index as $\alpha $.  We also have $O\left( {\mu ,C_\beta  } \right) = O\left( {\mu ,C_\alpha  } \right)$ for any $\beta  \in {\mathbf{D}}$, so ${\mathbf{D}}$ has a well defined orbit $O\left( {\mu ,{\mathbf{D}}} \right) = O\left( {\mu ,C_\alpha  } \right)$  for the action of $\mathfrak{S}_\mu  $ on $C\left( {i,r} \right)$.  There is also a well defined composition $\mu \left( {\mathbf{D}} \right)$ and a Young subgroup $\mathfrak{S}_{\mu \left( {\mathbf{D}} \right)} $ where $\mu \left( {\mathbf{D}} \right) = \mu \left( C \right)$ for any $C \in O\left( {\mu ,{\mathbf{D}}} \right)$.  Both $\mu \left( {\mathbf{D}} \right)$ and $\mathfrak{S}_{\mu \left( {\mathbf{D}} \right)} $ depend only on the orbit $O\left( {\mu ,{\mathbf{D}}} \right)$.  For a double coset ${\mathbf{D}} = \mathfrak{S}_\mu  \alpha \mathfrak{S}_\nu   \in \,_\mu  M_\nu  $, a left coset ${\mathbf{C}} \subseteq \mathfrak{S}_\mu  \alpha \mathfrak{S}_\nu  $ has the form ${\mathbf{C}} = \mathfrak{S}_\mu  \alpha  \cdot \pi  = \bigcup\limits_{C \in O\left( {\mu ,C_\alpha  } \right)} {\phi _C  \cdot \mathfrak{S}_{\mu \left( {C_\alpha  } \right)}  \cdot \sigma _\alpha   \circ \psi _{D_\alpha  } }  \cdot \pi $ for some $\pi  \in \mathfrak{S}_\nu  $, where the union is disjoint by  proposition \ref{p10.1}.   Then $n_L \left( {\mathbf{D}} \right) = \left| {\mathbf{C}} \right| = \left| {O\left( {\mu ,C_\alpha  } \right)} \right| \cdot \left| {\mathfrak{S}_{\mu \left( {C_\alpha  } \right)} } \right| = \left| {O\left( {\mu ,{\mathbf{D}}} \right)} \right| \cdot \left| {\mathfrak{S}_{\mu \left( {\mathbf{D}} \right)} } \right|$ , which depends only on the orbit $O_\mu   = O\left( {\mu ,{\mathbf{D}}} \right) = O\left( {\mu ,C_\alpha  } \right)$.  So we can define $n_L \left( {O_\mu  } \right) = n_L \left( {\mathbf{D}} \right)$ for any ${\mathbf{D}}$ with $O_\mu   = O\left( {\mu ,{\mathbf{D}}} \right)$.

      Similarly, ${\mathbf{D}}$ has a well defined orbit $O\left( {\nu ,{\mathbf{D}}} \right) = O\left( {\nu ,D_\alpha  } \right)$  for the action of $\mathfrak{S}_\nu  $ on $D\left( {M,i,r} \right)$ and there are a composition $\nu \left( {\mathbf{D}} \right)$ and a Young subgroup $\mathfrak{S}_{\nu \left( {\mathbf{D}} \right)} $ where $\nu \left( {\mathbf{D}} \right) = \nu \left( D \right)$ for any $D \in O\left( {\nu ,{\mathbf{D}}} \right)$.  Then $\nu \left( {\mathbf{D}} \right)$ , $\mathfrak{S}_{\nu \left( {\mathbf{D}} \right)} $ , and $n_R \left( {\mathbf{D}} \right) = \left| {O\left( {\nu ,{\mathbf{D}}} \right)} \right| \cdot \left| {\mathfrak{S}_{\nu \left( {\mathbf{D}} \right)} } \right|$ depend only on the orbit $O\left( {\nu ,{\mathbf{D}}} \right)$ and we can define $n_R \left( {O_\nu  } \right) = n_R \left( {\mathbf{D}} \right)$ for any ${\mathbf{D}}$ with $O_\nu   = O\left( {\nu ,{\mathbf{D}}} \right)$.

     Given an orbit pair $\left( {O_\mu  ,O_\nu  } \right) \in {\mathbf{O}}\left( {\mu ,\nu ,M,i} \right)$, define compositions $\mu \left( {O_\mu  } \right)$ and $\nu \left( {O_\nu  } \right)$ of $i$ and corresponding Young subgroups $\mathfrak{S}_{\mu \left( {O_\mu  } \right)} \,,\,\mathfrak{S}_{\nu \left( {O_\nu  } \right)} $, where $\mu \left( {O_\mu  } \right)\, = \mu \left( C \right)$ for any $C \in O_\mu  $ and  $\nu \left( {O_\nu  } \right)\, = \nu \left( D \right)$ for any $D \in O_\nu  $.  If ${}^{\mu \left( {O_\mu  } \right)}B^{\nu \left( {O_\nu  } \right)}  \subseteq \mathbb{Z}\left[ {\mathfrak{S}_i } \right]$ is the $\mathbb{Z}$-submodule of $B = \mathbb{Z}\left[ {\mathfrak{S}_i } \right]$ which is invariant under the action of $\mathfrak{S}_{\mu \left( {O_\mu  } \right)} $ on the left and $\mathfrak{S}_{\nu \left( {O_\nu  } \right)} $ on the right, then ${}^{\mu \left( {O_\mu  } \right)}B^{\nu \left( {O_\nu  } \right)} $ is a free $\mathbb{Z}$-module with basis $\left\{ {X\left( {\mathfrak{S}_{\mu \left( {O_\mu  } \right)} \sigma \,\mathfrak{S}_{\nu \left( {O_\nu  } \right)} } \right):\sigma  \in \mathfrak{S}_i } \right\}$.   Define a $\mathbb{Z}$- linear map $\Phi \left( {O_\mu  ,O_\nu  } \right):{}^{\mu \left( {O_\mu  } \right)}B^{\nu \left( {O_\nu  } \right)}  \to {}^{O_\mu  }A^{O_\nu  } $  by 
     \[\Phi \left( {O_\mu  ,O_\nu  } \right)\left( x \right) = \sum\limits_{C \in O_\mu  } {\sum\limits_{D \in O_\nu  } {\phi _C  \circ x \circ \psi _D } } .\]
       By proposition \ref{p10.1}, $\Phi \left( {O_\mu  ,O_\nu  } \right)$ is an isomorphism of free $\mathbb{Z}$-modules taking the basis elements $X\left( {\mathfrak{S}_{\mu \left( {O_\mu  } \right)} \sigma \,\mathfrak{S}_{\nu \left( {O_\nu  } \right)} } \right)$ for ${}^{\mu \left( {O_\mu  } \right)}B^{\nu \left( {O_\nu  } \right)} $ one to one onto the basis elements $X\left( {\mathfrak{S}_\mu  \alpha \mathfrak{S}_\nu  } \right)$ for ${}^{O_\mu  }A^{O_\nu  } $ where $\alpha  = \phi _C  \circ \sigma  \circ \psi _D $ for any $C \in O_\mu  ,D \in O_\nu  $.  Now as a $\mathbb{Z}$-module, ${}^{\mu \left( {O_\mu  } \right)}B^{\nu \left( {O_\nu  } \right)} $  can be identified with a direct summand of the standard Schur algebra $S_\mathbb{Z} \left( {r,n} \right)$.  The standard cellular basis for the Schur algebra $S_\mathbb{Z} \left( {r,n} \right)$ then yields a basis $\left\{ {\,_S C_T ^\lambda  } \right\}$  for ${}^{\mu \left( {O_\mu  } \right)}B^{\nu \left( {O_\nu  } \right)} $ where $\lambda $ is a partition of $i$, $S$ is a semistandard  $\lambda $ tableau of type $\mu \left( {O_\mu  } \right)$ and $T$ is a semistandard $\lambda $ tableau of type $\nu \left( {O_\nu  } \right)$.  Then $\left\{ {\Phi \left( {O_\mu  ,O_\nu  } \right)\left( {\,_S C_T ^\lambda  } \right)} \right\}$ gives a basis $B\left( {O_\mu  ,O_\nu  } \right)$ for ${}^{O_\mu  }A^{O_\nu  } $.   These piece together to give a basis for $\bar A$  which turns out to be a cell basis for $A_L ^\mathbb{Z} $ or $A_R ^\mathbb{Z} $.

     Let $\Lambda  = \bigcup\limits_{i \in I\left( M \right)} {\Lambda \left( i \right)} $ with the same partial order as for the cell algebra $\mathbb{Z}\left[ M \right]$ of section 7.  For $\lambda  \in \Lambda \left( i \right) \subset \Lambda $ define 
     \[L\left( \lambda  \right) = \left\{ {O_\mu  ,S:\mu  \in \Lambda \left( {r,n} \right)\,,\,O_\mu   \in {\mathbf{O}}\left( {\mu ,C\left( {i,r} \right)} \right),S \in SSt(\lambda ,\mu \left( {O_\mu  } \right)} \right\}\] where $SSt\left( {\lambda ,\mu \left( {O_\mu  } \right)} \right)$ is the set of semistandard $\lambda $
 tableaus of type $\mu \left( {O_\mu  } \right)$ and define \[R\left( \lambda  \right) = \left\{ {O_\nu  ,T:\nu  \in \Lambda \left( {r,n} \right)\,,\,O_\nu   \in {\mathbf{O}}\left( {\nu ,D\left( {M,i,r} \right)} \right),T \in SSt(\lambda ,\nu \left( {O_\nu  } \right)} \right\}\] where $SSt\left( {\lambda ,\nu \left( {O_\nu  } \right)} \right)$ is the set of semistandard $\lambda $ tableaus of type $\nu \left( {O_\nu  } \right)$.  Then for $\lambda  \in \Lambda \,,\,\left( {O_\mu  ,S} \right) \in L\left( \lambda  \right)\,,\,\left( {O_\nu  ,T} \right) \in R\left( \lambda  \right)$, define  ${}_{\left( {O_\mu  ,S} \right)}C_{\left( {O_\nu  ,T} \right)}^\lambda   = \Phi \left( {O_\mu  ,O_\nu  } \right)\left( {{}_SC_T^\lambda  } \right)$.    As just mentioned, elements of this type provide a $\mathbb{Z}$-basis for each direct summand ${}^{O_\mu  }A^{O_\nu  } $ and hence for all of  $\bar A$.   

     Note that if $M$ contains the zero map $z$ where $z\left( j \right) = 0$ for all $j$, then $\Lambda $ contains the "empty" partition $\lambda ^0 $ of index $0$.  For any partitions $\mu ,\nu $ the double coset $\mathfrak{S}_\mu  z\mathfrak{S}_\nu   = \left\{ z \right\} \subseteq {}^\mu A^\nu  $.  $C\left( {0,r} \right) = D\left( {M,0,r} \right) = \emptyset $, so for each partition there is one orbit $O_\mu  $ or $O_\nu  $.  Then $L\left( {\lambda ^0 } \right)$ contains one element $\left( {O_\mu  ,\emptyset } \right)$ for each partition $\mu $, and similarly for $R\left( {\lambda ^0 } \right)$.  Then for each pair of partitions $\mu ,\nu $ our basis contains an element $_{\left( {O_\mu  ,\emptyset } \right)} C^{\lambda ^0 } _{\left( {O_\nu  ,\emptyset } \right)}  = z \in {}^\mu A^\nu  $.

     We now show that the basis just described is a cell basis.  We first check that these basis elements have the left and right cell algebra properties (i) and (ii) for the ``ordinary'' product in $\bar A$.  We then check that the properties also hold for the products $ * _L {\text{ and }} * _R $ in $A_L ^\mathbb{Z} $ and $A_R ^\mathbb{Z} $.   

     Notice that each basis element ${}_{\left( {O_\mu  ,S} \right)}C_{\left( {O_\nu  ,T} \right)}^\lambda  $ for $\bar A$ is a sum of basis elements in the cell algebra $A = \mathbb{Z}\left[ M \right]$ of section 7 of the form $_{C,s} C_{t,D}^\lambda  $ for the same $\lambda $
 (where $s,t$ are standard tableaux of type $S,T$ , $C \in O_\mu  \,,\,D \in O_\nu  $).  As in section 7, let $A^\lambda  \,,\,\hat A^\lambda  $ be the ideals in the cell algebra $A = \mathbb{Z}\left[ M \right]$ and let $\bar A^\lambda  \,,\,\hat \bar A^\lambda  $ be the corresponding submodules of $\bar A$ (spanned by basis elements ${}_{\left( {O_\mu  ,S} \right)}C_{\left( {O_\nu  ,T} \right)}^\kappa  $  with $\kappa  \geqslant \lambda $ or $\kappa  > \lambda $ respectively).  Then $\hat A^\lambda   \cap {}^\mu A^\nu   = \hat \bar A^\lambda   \cap {}^\mu A^\nu  $ for any $\lambda ,\mu ,\nu $.   

\begin{lemma} \label{l10.3}
    Let ${}_{\left( {O_{\mu _i } ,S_i } \right)}C_{\left( {O_{\nu _i } ,T_i } \right)}^{\lambda _i }  \in {}^{O_{\mu _i } }A^{O_{\nu _i } } $ for  $i = 1,2$.  Then in $\bar A$,   ${}_{\left( {O_{\mu _1 } ,S_1 } \right)}C_{\left( {O_{\nu _1 } ,T_1 } \right)}^{\lambda _1 }  \cdot {}_{\left( {O_{\mu _2 } ,S_2 } \right)}C_{\left( {O_{\nu _2 } ,T_2 } \right)}^{\lambda _2 }  = \sum\limits_{\mu ',S'} {r \cdot \,\,{}_{\left( {O_{\mu '} ,S'} \right)}C_{\left( {O_{\nu _2 } ,T_2 } \right)}^{\lambda _2 } \,\,} \,\bmod \,\hat \bar A^{\lambda _2 } $ where the coefficients $r \in \mathbb{Z}$ are independent of  $O_{\nu _2 } {\text{ and }}T_2 $.
\end{lemma}

\begin{proof}
  Write ${}_{\left( {O_{\mu _2 } ,S_2 } \right)}C_{\left( {O_{\nu _2 } ,T_2 } \right)}^{\lambda _2 } $ as a sum of terms $_{C,s} C_{t,D}^{\lambda _2 } $  where $t$ is a standard tableau of type $T$ and $D \in O_{\nu _2 } $.  Then using property (i) for the cell algebra $A$, we have ${}_{\left( {O_{\mu _1 } ,S_1 } \right)}C_{\left( {O_{\nu _1 } ,T_1 } \right)}^{\lambda _1 }  \cdot {}_{\left( {O_{\mu _2 } ,S_2 } \right)}C_{\left( {O_{\nu _2 } ,T_2 } \right)}^{\lambda _2 }  = \sum\limits_{C',s'} {r \cdot \,\,{}_{C',s'}C_{t,D}^{\lambda _2 } } \,\,\,\bmod \,\hat A^{\lambda _2 } $ where the coefficients $r$ are independent of $D$ and $t$ and therefore of $O_{\nu _2 } {\text{ and }}T_2 $.   Also  ${}_{\left( {O_{\mu _1 } ,S_1 } \right)}C_{\left( {O_{\nu _1 } ,T_1 } \right)}^{\lambda _1 }  \cdot {}_{\left( {O_{\mu _2 } ,S_2 } \right)}C_{\left( {O_{\nu _2 } ,T_2 } \right)}^{\lambda _2 }  \in \,^{\mu _1 } A^{\nu _2 } $.  So the terms $_{C',s'} C_{t,D}^{\lambda _2 } $ must regroup into a linear combination of terms of the form ${}_{\left( {O_{\mu '} ,S'} \right)}C_{\left( {O_{\nu _2 } ,T_2 } \right)}^{\lambda _2 } $ .  Then using $\hat A^{\lambda _2 }  \cap {}^{\mu _1 }A^{\nu _2 }  = \hat \bar A^{\lambda _2 }  \cap {}^{\mu _1 }A^{\nu _2 } $
  gives the result. 
\end{proof}

Since the ${}_{\left( {O_\mu  ,S} \right)}C_{\left( {O_\nu  ,T} \right)}^\lambda  $ form a basis for $\bar A$ , linearity gives the following corollary.

\begin{corollary} \label{c10.1}
  For any $x \in \bar A$, \[x \cdot {}_{\left( {O_\mu  ,S} \right)}C_{\left( {O_\nu  ,T} \right)}^\lambda   = \sum\limits_{\mu ',S'} {r \cdot \,\,{}_{\left( {O_{\mu '} ,S'} \right)}C_{\left( {O_\nu  ,T_{} } \right)}^\lambda  } \,\,\,\bmod \,\hat \bar A^\lambda \] where $r = r\left( {x,\mu ,S,\mu ',S'} \right)$ is independent of $O_{\nu} ,T$.
\end{corollary}

     Similar arguments give the following results.   

\begin{lemma} \label{l10.4}
    Let ${}_{\left( {O_{\mu _i } ,S_i } \right)}C_{\left( {O_{\nu _i } ,T_i } \right)}^{\lambda _i }  \in {}^{O_{\mu _i } }A^{O_{\nu _i } } $
 for  $i = 1,2$.  Then in $\bar A$,   ${}_{\left( {O_{\mu _1 } ,S_1 } \right)}C_{\left( {O_{\nu _1 } ,T_1 } \right)}^{\lambda _1 }  \cdot {}_{\left( {O_{\mu _2 } ,S_2 } \right)}C_{\left( {O_{\nu _2 } ,T_2 } \right)}^{\lambda _2 }  = \sum\limits_{\nu ',T'} {r \cdot \,\,{}_{\left( {O_{\mu _1 } ,S_1 } \right)}C_{\left( {O_{\nu '} ,T'} \right)}^{\lambda _1 } } \,\,\,\bmod \,\hat \bar A^{\lambda _1 } $ where the coefficients $r \in \mathbb{Z}$ are independent of  $O_{\mu _1 } {\text{ and }}S_1 $.
\end{lemma}

\begin{corollary} \label{c10.2}
  For any $x \in \bar A$, \[{}_{\left( {O_\mu  ,S} \right)}C_{\left( {O_\nu  ,T} \right)}^\lambda   \cdot x = \sum\limits_{\nu ',T'} {r \cdot \,\,{}_{\left( {O_\mu  ,S} \right)}C_{\left( {O_{\nu '} ,T'} \right)}^\lambda  } \,\,\,\bmod \,\hat \bar A^\lambda \] where $r = r\left( {x,\nu ,T,\nu ',T'} \right)$ is independent of $O_{\mu} ,S$.
\end{corollary}

We now transfer our results to $A_L ^\mathbb{Z} $ and $A_R ^\mathbb{Z} $.  We need the following lemma.

\begin{lemma} \label{l10.5}
  Let $B =  \cup B\left( {O_\mu  ,O_\nu  } \right)$ be a basis for $\bar A$ where each $B\left( {O_\mu  ,O_\nu  } \right)$ is a basis for the direct summand ${}^{O_\mu  }A^{O_\nu  } $.  For  $b \in B\left( {O_\mu  ,O_\nu  } \right)$ define $n_L \left( b \right) = n_L \left( {O_\mu  } \right)$ and $n_R \left( b \right) = n_R \left( {O_\nu  } \right)$.  Assume that $b_1 b_2  = \sum\limits_{b \in B} {c(b_1 ,b_2 ,b)\,b} $  with structure constants $c\left( {b_1 ,b_2 ,b} \right) \in \mathbb{Z}$ (using the ``ordinary'' product in $\bar A$).  Then 
 $b_1  * _L b_2  = \sum\limits_{b \in B} {\frac{{n_L \left( b \right)}}
{{n_L \left( {b_1 } \right)n_L (b_2 )}}\,\,\,c(b_1 ,b_2 ,b)\,\,b} $  and $b_1  * _R b_2  = \sum\limits_{b \in B} {\frac{{n_R \left( b \right)}}
{{n_R \left( {b_1 } \right)n_R (b_2 )}}\,\,\,c(b_1 ,b_2 ,b)\,\,b} $.
\end{lemma}

\begin{proof}
  The result is true by definition for the standard basis $\left\{ {b_D  = X\left( D \right)} \right\}$.  Each $b \in B\left( {O_\mu  ,O_\nu  } \right)$ is a linear combination of standard basis vectors $b_D  \in B\left( {O_\mu  ,O_\nu  } \right)$ and each $b_D  \in {}^{O_\mu  }A^{O_\nu  } $ is a linear combination of the new $b \in B\left( {O_\mu  ,O_\nu  } \right)$.  Then since the values $n_L \left( b \right)\,,\,n_R \left( b \right)$ depend only on the orbits $O_\mu  \,,\,O_\nu  $ , the result for the new basis follows by linearity. 
\end{proof}

\begin{lemma} \label{l10.6}
   Let ${}_{\left( {O_{\mu _i } ,S_i } \right)}C_{\left( {O_{\nu _i } ,T_i } \right)}^{\lambda _i }  \in {}^{O_{\mu _i } }A^{O_{\nu _i } } $ for  $i = 1,2$.  Assume $\nu _1  = \mu _2 $.

\begin{enumerate} [\upshape(a)]
    \item  In $A_L ^\mathbb{Z} $,
    \[{}_{\left( {O_{\mu _1 } ,S_1 } \right)}C_{\left( {O_{\nu _1 } ,T_1 } \right)}^{\lambda _1 }  * _L {}_{\left( {O_{\mu _2 } ,S_2 } \right)}C_{\left( {O_{\nu _2 } ,T_2 } \right)}^{\lambda _2 }  = \sum\limits_{\mu ',S'} {r \cdot \,\,{}_{\left( {O_{\mu '} ,S'} \right)}C_{\left( {O_{\nu _2 } ,T_2 } \right)}^{\lambda _2 } } \,\,\,\bmod \,\hat \bar A^{\lambda _2 } \] where the coefficients $r \in \mathbb{Z}$ are independent of  $O_{\nu _2 } {\text{ and }}T_2 .$
      \item  In $A_L ^\mathbb{Z} $ ,
      \[{}_{\left( {O_{\mu _1 } ,S_1 } \right)}C_{\left( {O_{\nu _1 } ,T_1 } \right)}^{\lambda _1 }  * _L {}_{\left( {O_{\mu _2 } ,S_2 } \right)}C_{\left( {O_{\nu _2 } ,T_2 } \right)}^{\lambda _2 }  = \sum\limits_{\nu ',T'} {r \cdot \,\,{}_{\left( {O_{\mu _1 } ,S_1 } \right)}C_{\left( {O_{\nu '} ,T'} \right)}^{\lambda _1 } \,\,} \,\bmod \,\hat \bar A^{\lambda _1 } \] where the coefficients $r \in \mathbb{Z}$ are independent of  $O_{\mu _1 } {\text{ and }}S_1 .$
      \item  In $A_R ^\mathbb{Z} $,  
      \[{}_{\left( {O_{\mu _1 } ,S_1 } \right)}C_{\left( {O_{\nu _1 } ,T_1 } \right)}^{\lambda _1 }  * _R {}_{\left( {O_{\mu _2 } ,S_2 } \right)}C_{\left( {O_{\nu _2 } ,T_2 } \right)}^{\lambda _2 }  = \sum\limits_{\mu ',S'} {r \cdot \,\,{}_{\left( {O_{\mu '} ,S'} \right)}C_{\left( {O_{\nu _2 } ,T_2 } \right)}^{\lambda _2 } \,\,} \,\bmod \,\hat \bar A^{\lambda _2 } \] where the coefficients $r \in \mathbb{Z}$ are independent of  $O_{\nu _2 } {\text{ and }}T_2 .$
      \item  In $A_R ^\mathbb{Z} $ , 
      \[{}_{\left( {O_{\mu _1 } ,S_1 } \right)}C_{\left( {O_{\nu _1 } ,T_1 } \right)}^{\lambda _1 }  * _R {}_{\left( {O_{\mu _2 } ,S_2 } \right)}C_{\left( {O_{\nu _2 } ,T_2 } \right)}^{\lambda _2 }  = \sum\limits_{\nu ',T'} {r \cdot \,\,{}_{\left( {O_{\mu _1 } ,S_1 } \right)}C_{\left( {O_{\nu '} ,T'} \right)}^{\lambda _1 } } \,\,\,\bmod \,\hat \bar A^{\lambda _1 } \] where the coefficients $r \in \mathbb{Z}$ are independent of  $O_{\mu _1 } {\text{ and }}S_1 .$
\end{enumerate}
\end{lemma}

\begin{proof}
  Notice that a basis element $b = {}_{\left( {O_\mu  ,S} \right)}C_{\left( {O_\nu  ,T} \right)}^\lambda  $  is in  ${}^{O_\mu  }A^{O_\nu  } $, so in the notation of lemma \ref{l10.5} we have $n_L \left( b \right) = n_L \left( {O_\mu  } \right)\,,\,n_R \left( b \right) = n_R \left( {O_\nu  } \right)$.

For part (a), lemma \ref{l10.3} gives \[{}_{\left( {O_{\mu _1 } ,S_1 } \right)}C_{\left( {O_{\nu _1 } ,T_1 } \right)}^{\lambda _1 }  \cdot {}_{\left( {O_{\mu _2 } ,S_2 } \right)}C_{\left( {O_{\nu _2 } ,T_2 } \right)}^{\lambda _2 }  = \sum\limits_{\mu ',S'} {r' \cdot \,\,{}_{\left( {O_{\mu '} ,S'} \right)}C_{\left( {O_{\nu _2 } ,T_2 } \right)}^{\lambda _2 } \,\,} \,\bmod \,\hat \bar A^{\lambda _2 } \] where the coefficients $r' \in \mathbb{Z}$ are independent of  $O_{\nu _2 } {\text{ and }}T_2 $.  Then lemma \ref{l10.5} gives \[{}_{\left( {O_{\mu _1 } ,S_1 } \right)}C_{\left( {O_{\nu _1 } ,T_1 } \right)}^{\lambda _1 } \,\, * _L \,\,{}_{\left( {O_{\mu _2 } ,S_2 } \right)}C_{\left( {O_{\nu _2 } ,T_2 } \right)}^{\lambda _2 }  =\]\[ \sum\limits_{\mu ',S'} {\frac{{n_L \left( {O_{\mu '} } \right)}}
{{n_L \left( {O_{\mu _1 } } \right)n_L \left( {O_{\mu _2 } } \right)}} \cdot r' \cdot \,\,{}_{\left( {O_{\mu '} ,S'} \right)}C_{\left( {O_{\nu _2 } ,T_2 } \right)}^{\lambda _2 } } \,\,\,\bmod \,\hat \bar A^{\lambda _2 } .\]  Observing that $r = \frac{{n_L \left( {O_{\mu '} } \right)}}
{{n_L \left( {O_{\mu _1 } } \right)n_L \left( {O_{\mu _2 } } \right)}} \cdot r'$   is independent of $O_{\nu _2 } {\text{ and }}T_2 $ gives the result (a).

For (b), lemma \ref{l10.4} gives 
\[{}_{\left( {O_{\mu _1 } ,S_1 } \right)}C_{\left( {O_{\nu _1 } ,T_1 } \right)}^{\lambda _1 }  \cdot {}_{\left( {O_{\mu _2 } ,S_2 } \right)}C_{\left( {O_{\nu _2 } ,T_2 } \right)}^{\lambda _2 }  = \sum\limits_{\nu ',T'} {r' \cdot \,\,{}_{\left( {O_{\mu _1 } ,S_1 } \right)}C_{\left( {O_{\nu '} ,T'} \right)}^{\lambda _1 } } \,\,\,\bmod \,\hat \bar A^{\lambda _1 } \] where the coefficients $r' \in \mathbb{Z}$ are independent of  $O_{\mu _1 } {\text{ and }}S_1 .$  Then lemma \ref{l10.5} gives 
\[{}_{\left( {O_{\mu _1 } ,S_1 } \right)}C_{\left( {O_{\nu _1 } ,T_1 } \right)}^{\lambda _1 } \,\, * _L \,\, {}_{\left( {O_{\mu _2 } ,S_2 } \right)}C_{\left( {O_{\nu _2 } ,T_2 } \right)}^{\lambda _2 }  =\]\[ \sum\limits_{\nu ',T'} {\,\frac{{n_L \left( {O_{\mu _1 } } \right)}}
{{n_L \left( {O_{\mu _1 } } \right)n_L \left( {O_{\mu _2 } } \right)}} \cdot r' \cdot \,{}_{\left( {O_{\mu _1 } ,S_1 } \right)}C_{\left( {O_{\nu '} ,T'} \right)}^{\lambda _1 } } \,\,\,\bmod \,\hat \bar A^{\lambda _1 } .\]   Then $r = \frac{{n_L \left( {O_{\mu _1 } } \right)}}
{{n_L \left( {O_{\mu _1 } } \right)n_L \left( {O_{\mu _2 } } \right)}} \cdot r' = \frac{1}
{{n_L \left( {O_{\mu _2 } } \right)}} \cdot r'$ is independent of $O_{\mu _1 } {\text{ and }}S_1 $ and the result (b) follows.

Parts (c) and (d) are proved similarly. 
\end{proof}

\begin{proposition} \label{p10.2}
  $\left\{ {{}_{\left( {O_\mu  ,S} \right)}C_{\left( {O_\nu  ,T} \right)}^\lambda  } \right\}$ is a cell basis for both $A_L ^\mathbb{Z} $  and $A_R ^\mathbb{Z} $, which are therefore cell algebras.
\end{proposition}

\begin{proof}
  We have shown the $\left\{ {{}_{\left( {O_\mu  ,S} \right)}C_{\left( {O_\nu  ,T} \right)}^\lambda  } \right\}$ form a basis and the multiplication rules (i) and (ii) for a cell algebra then follow at once by linearity from lemma \ref{l10.6}. 
\end{proof}

\begin{corollary} \label{c10.3}
  For any domain $R$, the left and right generalized Schur algebras $LGS_R \left( {M,{\mathbf{G}}} \right) = R \otimes _\mathbb{Z} A_L ^\mathbb{Z} $ and $RGS_R \left( {M,{\mathbf{G}}} \right) = R \otimes _\mathbb{Z} A_R ^\mathbb{Z} $ are cell algebras with a cell basis $\left\{ {{}_{\left( {O_\mu  ,S} \right)}C_{\left( {O_\nu  ,T} \right)}^\lambda  } \right\}$.  
\end{corollary}

\section  {Irreducible modules for generalized Schur algebras}

     The cell basis $\left\{ {{}_{\left( {O_\mu  ,S} \right)}C_{\left( {O_\nu  ,T} \right)}^\lambda  } \right\}$ for the cell algebra $A_L ^\mathbb{Z} $  or $A_R ^\mathbb{Z} $ found above depends on the choice of an ordering of the orbits of  $\mathfrak{S}_\nu  $ acting on $P_r $ and of an ordering of the subsets in $P_r $ compatible with the ordering of the orbits.  We now choose orderings which will simplify the calculations of the brackets in these cell algebras.

     For $d \in P_r ,\,\,\left| d \right| = i$, define an increasing string of $i$ integers, $s\left( d \right)$, to be the $i$ elements of $d$ arranged in ascending order.  Then define a non-decreasing string of $i$ of positive integers, $s\left( {\nu ,d} \right)$, by replacing each $x \in s\left( d \right)$ by $b\left( x \right)$, where $x$ is in the $b\left( x \right)^{th} $ block $b_\nu  ^{b(x)} $ of the composition $\mathfrak{S}_\nu  $.  Finally, define a string of $r$ non-negative integers, $\bar s\left( {\nu ,d} \right)$, by adding $r - i$ zeroes to the end of $s\left( {\nu ,d} \right)$.  Note that $\bar s\left( {\nu ,d} \right)$ depends only on the $\mathfrak{S}_\nu  $-orbit of $d$.  In fact, $\bar s\left( {\nu ,d} \right) = \bar s\left( {\nu ,d'} \right) \Leftrightarrow d,d'$ are in the same $\mathfrak{S}_\nu  $-orbit.  We then get a total ordering of the $\mathfrak{S}_\nu  $-orbits by ordering the corresponding strings $\bar s\left( {\nu ,d} \right)$
 lexigraphically:  if $\bar s\left( {\nu ,d} \right)_j $ represents the $j$th element of the string, then $\bar s\left( {\nu ,d} \right) < \bar s\left( {\nu ,d'} \right) \Leftrightarrow$  for some $J$ between $1$ and $r$ we have $\bar s\left( {\nu ,d} \right)_j  = \bar s\left( {\nu ,d'} \right)_j \,$
for all $j < J$, while $\bar s\left( {\nu ,d} \right)_J  < \bar s\left( {\nu ,d'} \right)_J $.

    We then define our order on $P_r $ by :  $d < d'$ if $(1)\,\,\,\bar s\left( {\nu ,d} \right) < \bar s\left( {\nu ,d'} \right)$ or $(2)\,\,\,\bar s\left( {\nu ,d} \right) = \bar s\left( {\nu ,d'} \right)$
 (so $d,d'$ are in the same $\mathfrak{S}_\nu  $-orbit) and $s(d) < s(d')$ in lexicographical order.

     Note the following special cases of our ordering: \\
 The smallest set in $P_r $ is the empty set $\emptyset $. \\
 If $\left\{ a \right\},\left\{ b \right\}$ are one element sets in $P_r $, then $\left\{ a \right\} < \left\{ b \right\} \Leftrightarrow a < b$.\\
 If $\left\{ a \right\},d \in P_r $ and $d$ has more than one element with smallest element $b$, then $\left\{ a \right\} < d{\text{ if the }}\nu {\text{ - block containing }}a \leqslant {\text{the }}\nu {\text{ - block containing }}b$, while $\left\{ a \right\} > d{\text{ if the }}\nu {\text{ - block containing }}a > {\text{the }}\nu {\text{ - block containing }}b$.

 We will assume our cell bases are chosen with respect to these orderings.

	In this section we assume that $R = k$ is a field.  We will write $S_L \left( {M,k} \right)$ for the left generalized Schur algebra $LGS_k \left( {M,{\mathbf{G}}} \right) = k \otimes A_L ^\mathbb{Z} $ and $S_R \left( {M,k} \right)$ for the right generalized Schur algebra $RGS_k \left( {M,{\mathbf{G}}} \right) = k \otimes A_R ^\mathbb{Z} $.  For either of these algebras, if $\lambda  \in \Lambda $ then $\lambda  \in \Lambda \left( {i,n} \right)$ for some $i$ with $i \leqslant r \leqslant n .$  Recall that in these cell algebras $\Lambda _0 $ is the subset of $\Lambda $ consisting of $\lambda $ for which the bracket $\left\langle {_{O_\mu  ,S} C_{}^\lambda  ,C_{O_\nu  ,T}^\lambda  } \right\rangle $ is not identically zero.  By corollary \ref{c6.1} , there is one isomorphism class of irreducible modules for each $\lambda  \in \Lambda _0 $.  We will determine $\Lambda _0 $ when $M = \mathcal{T} _r $ or when $M$ contains the rook monoid $\Re _r $.

\begin{thm} \label{tt6.1}
    Let $k$ be a field of characteristic $0$ and let $M = \mathcal{T} _r $.  Then $\Lambda _0  = \Lambda $ for both $S_L \left( {\tau _r ,k} \right)$ and $S_R \left( {\tau _r ,k} \right)$.   Both $S_L \left( {\tau _r ,k} \right)$ and $S_R \left( {\tau _r ,k} \right)$ are quasi-hereditary algebras.   
\end{thm}

\begin{proof}
   Take any $\lambda  \in \Lambda $ with ${\text{index}}(\lambda ) = i > 0$.  Let $k$ be the largest index such that $\lambda _k  > 0$, so $\lambda _j  = 0\,,\,j > k$.  Let $\mu $ be the partition of $r$ where
   \[\mu _j  = 
      \begin{cases}
      \lambda _j      &\text{if $j \leqslant k$ }  \\
      1               &\text{if $j = k + 1,k + 2, \cdots ,k + (r - i)$ .}       
      \end{cases}
      \]

     Let $C = \left\{ {1,2, \cdots ,i} \right\}$.  Then $\mu \left( C \right) = \lambda  \in \Lambda \left( {i,n} \right)$.  Let $S$ be the semistandard $\lambda $-tableau of type $\mu \left( C \right)$ where $S_{j,l}  = j\,,\,1 \leqslant l \leqslant \lambda _j \,,\,j = 1,2, \cdots ,k$.  There is only one standard $\lambda $-tableau of type $S$, namely $s = id$, the identity in $\mathfrak{S}_i $.  The orbit $O\left( {C,\mu } \right) = \left\{ C \right\}$, so $\# O\left( {C,\mu } \right) = 1$.  Also, $\phi _C :\bar i \to \bar r$ is the identity $\phi _C (j) = j\,,\,j = 1,2, \cdots ,i$.

     Next let  $D = \left\{ {\left\{ 1 \right\},\left\{ 2 \right\}, \cdots ,\left\{ {i - 1} \right\},\left\{ {i,i + 1,i + 2, \cdots ,r} \right\}} \right\}$.  Then $\mu \left( D \right) \in \Lambda \left( {i,n} \right)$ is given by 
\[ \mu \left( D \right)_j  =
       \begin{cases}
       \lambda _j        &\text{if $j < k$ }  \\
       \lambda _k  - 1   &\text{if $j = k$ }  \\
       1                 &\text{if $j = k + 1$ }  \\
       0                 &\text{if $j > k + 1$ .}
       \end{cases}
\]
  Let $T$ be the semistandard $\lambda $-tableau of type $\mu \left( D \right)$ where 
\[ T_{j,l}  =
    \begin{cases}
    j          &\text{for $1 \leqslant l \leqslant \lambda _j \,,\,j = 1,2, \cdots ,k - 1$}  \\
    k          &\text{for $1 \leqslant l \leqslant \lambda _k  - 1\,,\,j = k$ }  \\
    k + 1      &\text{for  $l = \lambda _k \,,\,j = k$ .}
    \end{cases}
    \]
There is only one standard $\lambda $-tableau of type $T$, namely $t = id$, the identity in $\mathfrak{S}_i $.  Let $b_\mu  ^k $ be the $k$th block in the partition $\mu $.  For $a \in b_\mu  ^k $, define $D_a  = \left\{ {\left\{ 1 \right\},\left\{ 2 \right\}, \cdots ,\left\{ {i - 1} \right\},\left\{ i \right\},\left\{ {a,i + 1,i + 2, \cdots ,r} \right\}} \right\} - \left\{ {\left\{ a \right\}} \right\}$.  Then the orbit $O\left( {D,\mu } \right) = \left\{ {D_a :a \in b_\mu  ^k } \right\}$ and $\# O(D,\mu ) = \left| {b_\mu  ^k } \right| = \mu _k  = \lambda _k $.   Also $\psi _{D_a } :  \bar r \to \bar i$ is given by 
\[ \psi _{D_a } (j) =
    \begin{cases}
    j             &\text{if $j < a$ }  \\
    j - 1         &\text{if $a < j \leqslant i$ }  \\
    i             &\text{if $j = a$ or $j > i$ .}
 \end{cases}
\]
Then $\psi _{D_a }  \circ \phi _C :\bar i \to \bar i$ is a cyclic permutation 
$\sigma _a  = \left( {a,i,i - 1,i - 2, \cdots ,a + 1} \right) \in \mathfrak{S}_\lambda  .$

    We have $\,_S C_T^\lambda   = \,_s C_t^\lambda   = id \cdot r_\lambda   \cdot id = r_\lambda  $, so 
\[ \,_S C_T^\lambda   \circ \psi _{D_a }  \circ \phi _C  \circ \,_S C_T^\lambda   = r_\lambda   \cdot \sigma _a  \cdot r_\lambda   = r_\lambda   \cdot r_\lambda   = o\left( {\mathfrak{S}_\lambda  } \right)\,r_\lambda  .\]
Then writing $ b = _{O\left( {C,\mu } \right),S} C_{O(D,\mu ),T}^\lambda   = \phi _C  \cdot \,_S C_T^\lambda   \cdot \sum\limits_{a \in b_\mu  ^k } {\psi _{D_a } } ,$ compute 
\[\begin{aligned}
 b \cdot b  &= \phi _C  \cdot \,_S C_T^\lambda   \cdot \sum\limits_{a \in b_\mu  ^k } {\psi _{D_a } }  \cdot \phi _C  \cdot \,_S C_T^\lambda   \cdot \sum\limits_{a \in b_\mu  ^k } {\psi _{D_a } } \\
            &= \lambda _k  \cdot \phi _C  \cdot o\left( {\mathfrak{S}_\lambda  } \right) \cdot r_\lambda   \cdot \sum\limits_{a \in b_\mu  ^k } {\psi _{D_a } } \\
            &= \lambda _k  \cdot o\left( {\mathfrak{S}_\lambda  } \right) \cdot b .
\end{aligned} .\]
Then 
\[ b * _L b = \lambda _k  \cdot o\left( {\mathfrak{S}_\lambda  } \right) \cdot \frac{{n_L (b)}}
{{n_L \left( b \right)n_L (b)}} \cdot b = \lambda _k  \cdot o\left( {\mathfrak{S}_\lambda  } \right) \cdot \frac{1}
{{n_L \left( b \right)}} \cdot b = \lambda _k  \cdot b \]
 where
\[ n_L (b) = n_L \left( {O\left( {C,\mu } \right)} \right) = \# O\left( {C,\mu } \right) \cdot o(\mathfrak{S}_{\mu (C)} ) = 1 \cdot o(\mathfrak{S}_\lambda  ).\]
Similarly,
\[ b * _R b = \lambda _k  \cdot o\left( {\mathfrak{S}_\lambda  } \right) \cdot \frac{{n_R (b)}}
{{n_R \left( b \right)n_R (b)}} \cdot b = \lambda _k  \cdot o\left( {\mathfrak{S}_\lambda  } \right) \cdot \frac{1}
{{n_R \left( b \right)}} \cdot b = \lambda _k  \cdot b \]
where
\[ n_R (b) = n_R \left( {O\left( {D,\mu } \right)} \right) = \# O\left( {D,\mu } \right) \cdot o(\mathfrak{S}_{\mu (D)} ) = \lambda _k  \cdot o(\mathfrak{S}_\lambda  )/\lambda _k .\]

     Then computing the bracket (in either $S_L \left( {\tau _r ,k} \right)$ or $S_R \left( {\tau _r ,k} \right)$) we find that $\left\langle {C_{O(D,\mu ),T}^\lambda  \,,\,\,_{O\left( {C,\mu } \right),S} C_{}^\lambda  } \right\rangle  = \lambda _k  \ne 0$ (since characteristic of $k$ is $0$).  So $\lambda  \in \Lambda _0 $.

     By corollary \ref{c6.1}, a cell algebra with $\Lambda _0  = \Lambda $ is quasi-hereditary, so the proof is complete.
\end{proof}

Now assume $k$ is a field of characteristic $p > 0$.  For a partition $\lambda  \in \Lambda \left( i \right)\,,\,1 \leqslant i \leqslant r$, define an integer $k\left( {p,\lambda } \right) \geqslant 0$ to be the highest power of $p$ which divides $\lambda _j $ for every $j$.  (So for all $j$,  $\left. {p^{k(p,\lambda )} } \right|\lambda _j $, while for at least one $j$, $p^{k(p,\lambda ) + 1} $ does not divide $\lambda _j $.)  Then define $\Lambda _p  = \left\{ {\lambda  \in \Lambda :p^{k(p,\lambda )} {\text{ divides }}r - i{\text{, where }}i = {\text{index}}(\lambda )} \right\}$.  

\begin{lemma} \label{ll6.1}
  For a field $k$ of characteristic $p$ and the cell algebra $S_R \left( {\mathcal{T} _r ,k} \right)$ , $\Lambda _p  \subseteq \Lambda _0 $.
\end{lemma}

\begin{proof}
  Take any $\lambda  \in \Lambda _p $ of index $i$.  Let $m$ be the lowest nonzero row of $\lambda $, that is, assume $\lambda _m  > 0\,,\,\lambda _j  = 0{\text{ for }}j > m$.  Put $k = k\left( {p,\lambda } \right)$.  Let $a$ be the lowest row (i.e.,largest integer) such that $p^{k + 1} $  does not divide $\lambda _a $.  Since $\lambda  \in \Lambda _p $, $p^k $ divides $r - i$, so $q = (r - i)/p^k$ is an integer.
Define a composition $\mu $ of $r$ by splitting off the last $p^k $ elements of row $a$ of $\lambda $ and then adding $q$ additional rows of size $p^k $:
\[ \mu _j  =
    \begin{cases}
        \lambda _j           &\text{if $j < a$ }  \\
        \lambda _a  - p^k    &\text{if $j = a$ }  \\
        p^k                  &\text{if $j = a + 1$ }  \\
        \lambda _{a + l}     &\text{if $j = a + l + 1$ for $l = 1,2, \cdots ,m - a$ }  \\
        p^k                  &\text{if $j = m + 1 + l$ for $l = 1,2, \cdots ,q$ }
      \end{cases}
\]

     Let $C = \left\{ {1,2, \cdots ,i} \right\}$.  Then a composition of $i$ is given by $\mu \left( C \right)_j  = \mu _j \,,\,1 \leqslant j \leqslant m + 1$.  Let $S$ be the semistandard $\lambda $-tableau of type $\mu \left( C \right)$ where 
\[ S_{j,l}  =
   \begin{cases}
     j               &\text{for $1 \leqslant l \leqslant \lambda _j \,,\,j < a$ }  \\
     a               &\text{for $1 \leqslant l \leqslant \lambda _a  - p^k \,,\,j = a$ }  \\
     a + 1           &\text{for $\lambda _a  - p^k  < l \leqslant \lambda _a \,,\,j = a$ }  \\
     j + 1           &\text{for $1 \leqslant l \leqslant \lambda _j \,,\,a < j \leqslant m$ .}
    \end{cases}
\]
 There is only one standard $\lambda $-tableau of type $S$, namely $s = id$, the identity in $\mathfrak{S}_i $.  The orbit $O\left( {C,\mu } \right) = \left\{ C \right\}$, so $\# O\left( {C,\mu } \right) = 1$.  Also, $\phi _C :\bar i \to \bar r$ is the identity $\phi _C (j) = j\,,\,j = 1,2, \cdots ,i$.

     Now define $D \in D\left( {i,\tau _r ,r} \right)$ as follows:  $D$ contains $i - p^k $ single elements sets, one set $\left\{ l \right\}$ for each entry $l$ in rows $1$ through $a$ or rows $a + 2$ through $m + 1$ of $\mu $.  $D$ also contains $p^k $ sets with $q + 1$ entries:  for $1 \leqslant j \leqslant p^k $, the $j$ set $D_j $ contains the $j$th entry in row $a + 1$ and in each of the last $q$ rows of $\mu $.  As a composition of $i$, $\mu \left( D \right) = \mu \left( C \right)$, so we can take $T = S$ as a semistandard $\lambda $-tableau of type $\mu \left( D \right)$.  Then again there is only one standard $\lambda $-tableau of type $T$, namely $t = id$, the identity in $\mathfrak{S}_i $.  

     To define the orbit space $O\left( {D,\mu } \right)$, let $\mathfrak{S}_{\mu _{m + 1 + j} }  \subseteq \mathfrak{S}_\mu  $ be the group of permutations of  row $m + 1 + j$ of $\mu $.  For each $1 \leqslant j \leqslant q$, $\mathfrak{S}_{\mu _{m + 1 + j} }  \cong \mathfrak{S}_{p^k } $.  Let $G = \prod\limits_{j = 1}^q {\mathfrak{S}_{\mu _{m + 1 + j} } }  \subseteq \mathfrak{S}_\mu  $
 and for $\sigma  \in G$ let $D_\sigma   = D\sigma  \in D\left( {i,\tau _r ,r} \right)$.  Then $O\left( {D,\mu } \right) = \left\{ {D_\sigma  :\sigma  \in G} \right\}$.  Then $\# O\left( {D,\mu } \right) = o\left( G \right) = o\left( {\left( {\mathfrak{S}_{p^k } } \right)^q } \right) = \left( {p^k !} \right)^q $.  

      Notice that with our choice of an ordering of the subsets of $\bar r$, we get $\psi _{D_\sigma  } \left( j \right) = j\,,\,1 \leqslant j \leqslant i$, for any $\sigma  \in G$.  Then $\psi _{D_\sigma  }  \circ \phi _C  = id$, the identity mapping $\bar i \to \bar i$.  We have $ \,_S C_T^\lambda   = \,_s C_t^\lambda   = id \cdot r_\lambda   \cdot id = r_\lambda ,$ so
\[ \,_S C_T^\lambda   \circ \psi _{D_\sigma  }  \circ \phi _C  \circ \,_S C_T^\lambda   = r_\lambda   \cdot id \cdot r_\lambda   = r_\lambda   \cdot r_\lambda   = o\left( {\mathfrak{S}_\lambda  } \right)\,r_\lambda .\]
 Then writing $ b = _{O\left( {C,\mu } \right),S} C_{O(D,\mu ),T}^\lambda   = \phi _C  \cdot \,_S C_T^\lambda   \cdot \sum\limits_{\sigma  \in G} {\psi _{D_\sigma  } } ,$ compute
\[\begin{aligned}
 b \cdot b &= \phi _C  \cdot \,_S C_T^\lambda   \cdot \sum\limits_{\sigma  \in G} {\psi _{D_\sigma  } }  \cdot \phi _C  \cdot \,_S C_T^\lambda   \cdot \sum\limits_{\sigma  \in G} {\psi _{D_\sigma  } } \\
           &= o\left( G \right) \cdot \phi _C  \cdot o\left( {\mathfrak{S}_\lambda  } \right) \cdot r_\lambda   \cdot \sum\limits_{\sigma  \in G} {\psi _{D_\sigma  } } \\
           &= o\left( G \right) \cdot o\left( {\mathfrak{S}_\lambda  } \right) \cdot b .
\end{aligned} \]
Then
\[\begin{aligned}
 b * _R b &= o\left( G \right) \cdot o\left( {\mathfrak{S}_\lambda  } \right) \cdot \frac{{n_R (b)}}
{{n_R \left( b \right)n_R (b)}} \cdot b \\
          &= o\left( G \right) \cdot o\left( {\mathfrak{S}_\lambda  } \right) \cdot \frac{1}
{{n_R \left( b \right)}} \cdot b \\
          &= \frac{{o\left( G \right)o\left( {\mathfrak{S}_\lambda  } \right)}}
{{\# O\left( {D,\mu } \right)o\left( {\mathfrak{S}_{\mu \left( D \right)} } \right)}} \cdot b 
\end{aligned}
\]
where $ n_R (b) = n_R \left( {O\left( {D,\mu } \right)} \right) = \# O\left( {D,\mu } \right) \cdot o(\mathfrak{S}_{\mu (D)} ) .$ Since $\# O\left( {D,\mu } \right) = o\left( G \right)$ and $\frac{{o\left( {\mathfrak{S}_\lambda  } \right)}}
{{o\left( {\mathfrak{S}_{\mu (D)} } \right)}} = \frac{{\lambda _a !}}
{{\left( {\lambda _a  - p^k } \right)!p^k !}} = \left( {\begin{array}{*{20}c}
   {\lambda _a }  \\
   {p^k }  \\

 \end{array} } \right)$, we get $b * _R b = \left( {\begin{array}{*{20}c}
   {\lambda _a }  \\
   {p^k }  \\

 \end{array} } \right) \cdot b$.

    Computing the bracket gives $\left\langle {C_{O(D,\mu ),T}^\lambda  \,,\,\,_{O\left( {C,\mu } \right),S} C_{}^\lambda  } \right\rangle  = \left( {\begin{array}{*{20}c}
   {\lambda _a }  \\
   {p^k }  \\

 \end{array} } \right)$.  Since $p^k $ divides $\lambda _a $ but $p^{k + 1} $ does not, it is easily checked that $\left( {\begin{array}{*{20}c}
   {\lambda _a }  \\
   {p^k }  \\

 \end{array} } \right)$ is not congruent to 0 mod $p$.   So the bracket is not identically zero in $S_R \left( {\tau _r ,k} \right)$  and $\lambda  \in \Lambda _0 $ as desired.  
\end{proof}

     We claim that in fact $\Lambda _p  = \Lambda _0 $, that is, that every irreducible representation of  $S_R \left( {\mathcal{T} _r ,k} \right)$
 corresponds to some $\lambda  \in \Lambda _p $.  In \cite{May} or \cite{May4}, a parameterization of the isomorphism classes of  irreducible representations of  $S_R \left( {\mathcal{T} _r ,k} \right)$ is given.  There is one isomorphism class corresponding to the following set of data:  i) a set of nonnegative integers $s_m ,s_{m + 1} , \cdots ,s_M $ with $s_m  > 0,m \geqslant 0$ such that $r = s_m p^m  + s_{m + 1} p^{m + 1}  +  \cdots  + s_M p^M $
, ii) a $p$-restricted partition of $s_i $ for each $m < i \leqslant M$, and iii) a $p$-restricted partition of $i$ for some integer $1 \leqslant i \leqslant s_m $.  (The index of the corresponding irreducible is $r - (s_m  - i)p^m $.)  We will show that each such set of data corresponds with a unique element $\lambda  \in \Lambda _p $, so that  the number of isomorphism classes is less than or equal to $\# \Lambda _p $.   But we know that the number of isomorphism classes is $\# \Lambda _0 $ and by the lemma $\Lambda _p  \subseteq \Lambda _0 $.  So we must have $\Lambda _p  = \Lambda _0 $ . 

     We will need a certain ``decomposition'' operation on partitions.  For any integer $n > 0$, let $\lambda $ be a partition of $n$ with $R$ non-zero parts, $\lambda _1  + \lambda _2  +  \cdots  + \lambda _R  = n$.  For $1 \leqslant i \leqslant R$ define the row length differences $\Delta _i  = \lambda _i  - \lambda _{i + 1} $.  Then define an integer $k\left( \lambda  \right) \geqslant 0$ to be the highest power of $p$ which is less than or equal to at least one $\Delta _i $.  Then we can find nonnegative integers $q_i ,r_i $ such that $\Delta _i  = q_i p^{k\left( \lambda  \right)}  + r_i $ where  each $r_i  < p^{k(\lambda )} $ , each $q_i  < p$, and at least one $q_i  > 0$.  Define $s\left( \lambda  \right) = \sum\limits_{i = 1}^R {i \cdot q_i } $.  We will construct a $p$-restricted partition of $s\left( \lambda  \right)$ and a partition $\bar \lambda $ of  $n - s\left( \lambda  \right)p^{k\left( \lambda  \right)} $ with $k\left( {\bar \lambda } \right) < k\left( \lambda  \right)$.  Notice that there are $\Delta _i  = q_i p^{k\left( \lambda  \right)}  + r_i $ columns of height $i$ in $\lambda $.  We break $\lambda $ into two partitions $\lambda _q ,\lambda _r $ by placing $q_i p^{k\left( \lambda  \right)} $ columns of height $i$ in the first partition and $r_i $ columns of height $i$ in the second.  Then $\lambda _q $ is a partition of $i \cdot q_i p^{k(\lambda )}  = s\left( \lambda  \right)p^{k\left( \lambda  \right)} $ with row differences $q_i p^{k\left( \lambda  \right)} $.  By replacing each set of  $p^{k\left( \lambda  \right)} $ consecutive boxes in a row of $\lambda _q $ by a single box, we obtain a partition of $s\left( \lambda  \right)$ with row differences $q_i  < p$, i.e, a $p$-restricted partition of $s\left( \lambda  \right)$.  The second partition $\lambda _r $ is a partition of $n - s\left( \lambda  \right)p^{k\left( \lambda  \right)} $ with row differences $r_i  < p^{k(\lambda )} $.  Then $k\left( {\lambda _r } \right) < k\left( \lambda  \right)$ and we take $\bar \lambda  = \lambda _r $.  

     We can now replace $\lambda $ with $\lambda _r $ and iterate the construction until we reach a case when all $r_i $ are 0.  The result is a sequence of nonnegative integers $s_m ,s_{m + 1} , \cdots ,s_M $ with $s_m  > 0 , m \geqslant 0$ such that $n = s_m p^m  + s_{m + 1} p^{m + 1}  +  \cdots  + s_M p^M $ and a $p$-restricted partition of $s_i $ for each $m \leqslant i \leqslant M$.  By replacing each box in the partition of  $s_i $ by a row of $p^i $ boxes and then joining the resulting partitions (taking the union of the boxes in each row of each partition) we recover uniquely the original partition $\lambda $ of $n$.  Notice that $k\left( {p,\lambda } \right) = m$.

     Now take any isomorphism class of irreducible $S_R \left( {\tau _r ,k} \right)$ modules and consider the unique corresponding data
  i) nonnegative integers $s_m ,s_{m + 1} , \cdots ,s_M $ with $s_m  > 0$ such that $r = s_m p^m  + s_{m + 1} p^{m + 1}  +  \cdots  + s_M p^M $, 
ii) a $p$-restricted partition of $s_i $
for each $m < i \leqslant M$, and 
iii) a $p$-restricted partition of $s'_m $ for some integer $1 \leqslant s'_m  \leqslant s_m $. 

Then  $s'_m p^m  + s_{m + 1} p^{m + 1}  +  \cdots  + s_M p^M  = r - \left( {s_m  - s'_m } \right)p^m $, so our construction gives a unique partition $\lambda $ of  $r - \left( {s_m  - s'_m } \right)p^m $ with $k\left( {p,\lambda } \right) = m$.  Then $p^{k\left( {p,\lambda } \right)}  = p^m $ divides  $r - {\text{index}}(\lambda ) = r - (r - \left( {s_m  - s'_m } \right)p^m ) = \left( {s_m  - s'_m } \right)p^m $, so $\lambda  \in \Lambda _p $.  So the number of isomorphism classes is $ \leqslant \# \Lambda _p $ as desired.      

     As remarked above, this proves the following result.

\begin{thm} \label{tt6.2} 
  If $k$ is a field of characteristic $p$, then $\Lambda _p  = \Lambda _0 $  in $S_R \left( {\mathcal{T} _r ,k} \right)$. 
\end{thm}

\begin{corollary} \label{cc6.1}
  If $k$ is a field of characteristic $p$ and $r = ap^l $ for $1 \leqslant a < p$ and some $l = 0,1,2, \cdots $, then $S_R \left( {\mathcal{T} _r ,k} \right)$ is quasi-hereditary.
\end{corollary}

\begin{proof}
By corollary \ref{c6.1}, we must show that $\Lambda  = \Lambda _0 $ , that is, that any $\lambda  \in \Lambda $ is actually in $\Lambda _p  = \Lambda _0 $.  So suppose  $\lambda $ is a partition of $i$ for some $0 < i \leqslant r$.  Put $k = k\left( {p,\lambda } \right)$.  Since $p^k $ divides $\lambda _j $ for every $j$,  $p^k $ divides $i = \sum\limits_j {\lambda _j } $, say $i = bp^k $ for some $b > 0$.  Now $bp^k  = i \leqslant r = ap^l  < p^{l + 1} $ (since $a < p$), so $k \leqslant l$.  Then $r - i = ap^l  - bp^k  = \left( {ap^{l - k}  - b} \right)p^k $, so $p^k $ divides $r - i$ and $\lambda  \in \Lambda _p $ as desired.
\end{proof}

Now consider $S_L \left( {\mathcal{T} _r ,k} \right)$ for characteristic $p$.  Define 
\[\Lambda _{L,p}  = \left\{ {\lambda  \in \Lambda :p{\text{ does not divide }}\lambda _j {\text{ for at least one }}j} \right\} \cup \Lambda \left( r \right).\]
  
\begin{lemma} \label{ll6.2}
 For a field $k$ of characteristic $p$ and the cell algebra $S_L \left( {\mathcal{T} _r ,k} \right)$ , $\Lambda _{L,p}  \subseteq \Lambda _0 $.
\end{lemma}

\begin{proof}
  First suppose $\lambda  \in \Lambda (r)$, a partition of maximal index $r$.  Then take $\mu  = \lambda $ as a composition of $r$ and let $C = \left\{ {1, \cdots ,r} \right\}$.   Then $\mu \left( C \right) = \mu  = \lambda $ and a semi-standard $\lambda $-tableau $S$ of type $\mu \left( C \right)$ is given by $S_{j,l}  = j\,,\,1 \leqslant l \leqslant j$.   There is only one standard $\lambda $-tableau of type $S$, namely $s = id$, the identity in $\mathfrak{S}_i $.  The orbit $O\left( {C,\mu } \right) = \left\{ C \right\}$, so $\# O\left( {C,\mu } \right) = 1$.  Also, $\phi _C :\bar r \to \bar r$ is the identity $\phi _C (j) = j\,,\,j = 1,2, \cdots ,r$.

     Define $D \in D\left( {r,\tau _r ,r} \right)$ by $D = \left\{ {\left\{ 1 \right\},\left\{ 2 \right\}, \cdots ,\left\{ r \right\}} \right\}$.  As a composition of $r$, $\mu \left( D \right) = \mu \left( C \right) = \lambda $, so we can take $T = S$ as a semistandard $\lambda $
-tableau of type $\mu \left( D \right)$.  Then again there is only one standard $\lambda $-tableau of type $T$, namely $t = id$, the identity in $\mathfrak{S}_r $.  

     We have $O\left( {D,\mu } \right) = \left\{ D \right\}$, $\# O\left( {D,\mu } \right) = 1$,  $\psi _D \left( j \right) = j\,,\,1 \leqslant j \leqslant r$.  Then $\psi _D  \circ \phi _C  = id:\bar r \to \bar r$.  We have $\,_S C_T^\lambda   = \,_s C_t^\lambda   = id \cdot r_\lambda   \cdot id = r_\lambda  $, so $\,_S C_T^\lambda   \circ \psi _D  \circ \phi _C  \circ \,_S C_T^\lambda   = r_\lambda   \cdot id \cdot r_\lambda   = r_\lambda   \cdot r_\lambda   = o\left( {\mathfrak{S}_\lambda  } \right)\,r_\lambda  $.  Then writing $b = _{O\left( {C,\mu } \right),S} C_{O(D,\mu ),T}^\lambda   = \phi _C  \cdot \,_S C_T^\lambda   \cdot \psi _D $, compute  $b \cdot b = \phi _C  \cdot \,_S C_T^\lambda   \cdot \psi _D  \cdot \phi _C  \cdot \,_S C_T^\lambda   \cdot \psi _D  = \phi _C  \cdot o\left( {\mathfrak{S}_\lambda  } \right) \cdot r_\lambda   \cdot \psi _D  = o\left( {\mathfrak{S}_\lambda  } \right) \cdot b$.  Then $b * _L b = o\left( {\mathfrak{S}_\lambda  } \right) \cdot \frac{{n_L (b)}}
{{n_L \left( b \right)n_L (b)}} \cdot b = o\left( {\mathfrak{S}_\lambda  } \right) \cdot \frac{1}
{{n_L \left( b \right)}} \cdot b = \frac{{o\left( {\mathfrak{S}_\lambda  } \right)}}
{{\# O\left( {C,\mu } \right)o\left( {\mathfrak{S}_{\mu \left( C \right)} } \right)}} \cdot b$ where $n_L (b) = n_L \left( {O\left( {C,\mu } \right)} \right) = \# O\left( {C,\mu } \right) \cdot o(\mathfrak{S}_{\mu (C)} )$.  Since $\# O\left( {C,\mu } \right) = 1$ and $o\left( {\mathfrak{S}_{\mu \left( C \right)} } \right) = o\left( {\mathfrak{S}_\lambda  } \right)$, we get $b * _L b = b$.  Computing the bracket gives $\left\langle {C_{O(D,\mu ),T}^\lambda  \,,\,\,_{O\left( {C,\mu } \right),S} C_{}^\lambda  } \right\rangle  = 1 \ne 0$.  So the bracket is not identically zero in $S_L \left( {\tau _r ,k} \right)$  and $\lambda  \in \Lambda _0 $ as desired.

     Now take any $\lambda  \in \Lambda _{L,p} $ of index $i < r$.  Let $m$ be the lowest nonzero row of $\lambda $, that is, assume $\lambda _m  > 0\,,\,\lambda _j  = 0{\text{ for }}j > m$.  Let $a$ be the largest integer such that $p$ does not divide $\lambda _a $.  Define a composition $\mu $ of $r$ by splitting off the last element of row $a$ of $\lambda $ and also adding $r - i$ additional rows of length 1:
\[ \mu _j  = 
    \begin{cases}
     \lambda _j            &\text{if $j < a$ }  \\
     \lambda _a  - 1       &\text{if $j = a$ }  \\
     1                     &\text{if $j = a + 1$ }  \\
     \lambda _{j - 1}      &\text{if $a + 2 \leqslant j \leqslant m + 1$ }  \\
     1                    &\text{if $m + 2 \leqslant j \leqslant \left({m + 1}\right)+(r-i)$ .}
   \end{cases}
\]

     Let $C = \left\{ {1,2, \cdots ,i} \right\}$.  Then a composition of $i$ is given by $\mu \left( C \right)_j  = \mu _j \,,\,1 \leqslant j \leqslant m + 1$.  Let $S$ be the semistandard $\lambda $-tableau of type $\mu \left( C \right)$ where
\[ S_{j,l}  =
      \begin{cases}
       j                &\text{for $1 \leqslant l \leqslant \lambda _j \,,\,j < a$ }  \\
       a                &\text{for $1 \leqslant l \leqslant \lambda _a  - 1\,,\,j = a$ }  \\
       a + 1            &\text{for $l = \lambda _a \,,\,j = a$ }  \\
       j + 1        &\text{for $1 \leqslant l \leqslant \lambda _j \,,\,a < j \leqslant m$ .}  
       \end{cases}
\]
There is only one standard $\lambda $-tableau of type $S$, namely $s = id$, the identity in $\mathfrak{S}_i $.  The orbit $O\left( {C,\mu } \right) = \left\{ C \right\}$, so $\# O\left( {C,\mu } \right) = 1$.  Also, $\phi _C :\bar i \to \bar r$ is the identity $\phi _C (j) = j\,,\,j = 1,2, \cdots ,i$.

     Now define $D \in D\left( {i,\tau _r ,r} \right)$ as follows:  Let $x$ be the last element in row $\lambda _a $, that is, $x = \lambda _1  + \lambda _2  +  \cdots  + \lambda _a $.  $D$ contains $i - 1$ single elements sets, one set $\left\{ l \right\}$ for every $1 \leqslant l \leqslant i$ except $l = x$.  $D$ also contains one set with $r - i + 1$
 elements: $\left\{ {x,i + 1,i + 2, \cdots ,r} \right\}$.  As a composition of $i$, $\mu \left( D \right) = \mu \left( C \right)$, so we can take $T = S$ as a semistandard $\lambda $-tableau of type $\mu \left( D \right)$.  Then again there is only one standard $\lambda $
-tableau of type $T$, namely $t = id$, the identity in $\mathfrak{S}_i $.  

     We have $O\left( {D,\mu } \right) = \left\{ D \right\}$, $\# O\left( {D,\mu } \right) = 1$,  $\psi _D \left( j \right) =
     \begin{cases}
        j     &\text{if $1 \leqslant j \leqslant i$ }  \\
        x     &\text{if $j > i$ }
      \end{cases}.$  Then $\psi _D  \circ \phi _C  = id:\bar i \to \bar i$.  We have $\,_S C_T^\lambda   = \,_s C_t^\lambda   = id \cdot r_\lambda   \cdot id = r_\lambda  $, so $\,_S C_T^\lambda   \circ \psi _D  \circ \phi _C  \circ \,_S C_T^\lambda   = r_\lambda   \cdot id \cdot r_\lambda   = r_\lambda   \cdot r_\lambda   = o\left( {\mathfrak{S}_\lambda  } \right)\,r_\lambda  $.  Then writing $b = _{O\left( {C,\mu } \right),S} C_{O(D,\mu ),T}^\lambda   = \phi _C  \cdot \,_S C_T^\lambda   \cdot \psi _D ,$
 compute
\[
\begin{aligned}
 b \cdot b &= \phi _C  \cdot \,_S C_T^\lambda   \cdot \psi _D  \cdot \phi _C  \cdot \,_S C_T^\lambda   \cdot \psi _D \\
            &= \phi _C  \cdot o\left( {\mathfrak{S}_\lambda  } \right) \cdot r_\lambda   \cdot \psi _D  \\
             &= o\left( {\mathfrak{S}_\lambda  } \right) \cdot b .
\end{aligned}
\]
  Then \[b * _L b = o\left( {\mathfrak{S}_\lambda  } \right) \cdot \frac{{n_L (b)}}
{{n_L \left( b \right)n_L (b)}} \cdot b = o\left( {\mathfrak{S}_\lambda  } \right) \cdot \frac{1}
{{n_L \left( b \right)}} \cdot b = \frac{{o\left( {\mathfrak{S}_\lambda  } \right)}}
{{\# O\left( {C,\mu } \right)o\left( {\mathfrak{S}_{\mu \left( C \right)} } \right)}} \cdot b \] where $n_L (b) = n_L \left( {O\left( {C,\mu } \right)} \right) = \# O\left( {C,\mu } \right) \cdot o(\mathfrak{S}_{\mu (C)} )$.  Since $\# O\left( {C,\mu } \right) = 1$ and $\frac{{o\left( {\mathfrak{S}_\lambda  } \right)}}
{{o\left( {\mathfrak{S}_{\mu (C)} } \right)}} = \frac{{\lambda _a !}}
{{\left( {\lambda _a  - 1} \right)!1!}} = \lambda _a $, we get $b * _L b = \lambda _a  \cdot b$.

    Computing the bracket gives $\left\langle {C_{O(D,\mu ),T}^\lambda  \,,\,\,_{O\left( {C,\mu } \right),S} C_{}^\lambda  } \right\rangle  = \lambda _a $.  Since $p$ does not divide $\lambda _a $ (by definition),  $\lambda _a  \ne 0$ in $k$.   So the bracket is not identically zero in $S_L \left( {\tau _r ,k} \right)$  and $\lambda  \in \Lambda _0 $ as desired.
\end{proof}

    We claim that $\Lambda _{L,p}  = \Lambda _0 $, that is, that every irreducible representation of  $S_L \left( {\mathcal{T} _r ,k} \right)$ corresponds to some $\lambda  \in \Lambda _{L,p} $.  \cite{May4} gives the following parameterization of the isomorphism classes of irreducible representations of $S_L \left( {\mathcal{T} _r ,k} \right)$.  There is one isomorphism class corresponding to the following set of data:  i) a set of nonnegative integers $s_m ,s_{m + 1} , \cdots ,s_M $ with $m \geqslant 0,\,s_m  > 0$ such that $r = s_m p^m  + s_{m + 1} p^{m + 1}  +  \cdots  + s_M p^M $, ii) a $p$-restricted partition of $s_i $ for each $m < i \leqslant M$,  iii) a $p$-restricted partition of $s_m $ if $m > 0$ or a $p$-restricted partition of $i$ for some integer $1 \leqslant i \leqslant s_0 $ if $m = 0$.  (The index of the corresponding irreducible is $r - (s_0  - i)$.)  We will show that each such set of data corresponds with a unique element $\lambda  \in \Lambda _{L,p} $, so that  the number of isomorphism classes is less than or equal to $\# \Lambda _{L,p} $.   We know that the number of isomorphism classes is $\# \Lambda _0 $ and by the lemma $\Lambda _{L,p}  \subseteq \Lambda _0 $, so we must have $\Lambda _{L,p}  = \Lambda _0 $ . 

     Consider first a set of data for the case $m > 0$.  Then we have $r = s_m p^m  + s_{m + 1} p^{m + 1}  +  \cdots  + s_M p^M $ and a $p$-restricted partition of $s_m $ for all $m$.  So by the construction preceding theorem \ref{tt6.2} there is a unique partition $\lambda $ of index $r$ corresponding to the data, and since the index of $\lambda $ is $r$ we have $\lambda  \in \Lambda _{L,p} $.  Next consider a set of data for the case $m = 0$.  Putting $s'_0  = i$ we have $s'_0  + s_1 p^1  +  \cdots  + s_M p^M  = r - (s_0  - s'_0 )$ with $p$-restricted partitions of $s'_0  = i$ and all $s_i ,\,\,i > 0$.  The result is a unique partition $\lambda $ of $r - (s_0  - i)$ with $k\left( {p,\lambda } \right) = m = 0$.  But $k\left( {p,\lambda } \right) = 0$ means that at least one row length $\lambda _j $ of $\lambda $ is not divisible by $p^{k\left( {p,\lambda } \right) + 1}  = p$, that is, that $\lambda  \in \Lambda _{L,p} $.  So corresponding to each set of data there is a unique element $\lambda  \in \Lambda _{L,p} $ as desired.  As remarked above, this proves the following theorem. 

\begin{thm} \label{tt6.3}
  If $k$ is a field of characteristic $p$, then $\Lambda _{L,p}  = \Lambda _0 $ in $S_L \left( {\mathcal{T} _r ,k} \right)$.
\end{thm}

     Notice that if $p \geqslant r$ then $\Lambda _0  = \Lambda _{L,p}  = \Lambda $ and $S_L \left( {\mathcal{T} _r ,k} \right)$ is quasi-hereditary.  However, when $r > p$ we have $\Lambda _0  = \Lambda _{L,p}  \ne \Lambda $ and $S_L \left( {\mathcal{T} _r ,k} \right)$ is not quasi-hereditary for the given poset structure $\Lambda $.

     Now consider the case when $M$ contains the rook monoid $\Re _r $.  

\begin{thm}  \label{tt6.4}
  Assume $M$ contains the rook monoid $\Re _r $.  Then for any field $k$, $\Lambda _0  = \Lambda $ for  both cell algebras $S_L \left( {M,k} \right)$ and $S_R \left( {M,k} \right)$.  Both $S_L \left( {M,k} \right)$ and $S_R \left( {M,k} \right)$ are quasi-hereditary.
\end{thm}

Remark:  If $M$ is just the rook monoid, $M = \Re _r $, then $S_L \left( {M,k} \right)$ and $S_R \left( {M,k} \right)$ are both actually cellular algebras and are anti-isomorphic as algebras.
 
\begin{proof}
  Take any partition $\lambda $ of $i$ , $0 \leqslant i \leqslant r$.  We must show that $\lambda  \in \Lambda _0 $.  Let $m \geqslant 0$
 be the smallest integer such that  $\lambda _j  = 0{\text{ for }}j > m$
.  Define a composition $\mu $ of $r$ by adding $r - i$ rows of length 1 to $\lambda $.  So 
\[ \mu _j  =
     \begin{cases}
        \lambda _j     &\text{if $j \leqslant m$ }  \\
        1              &\text{if $m + 1 \leqslant j \leqslant m + r - i$ }  \\
        0              &\text{if $j > m + r - i$ .}
      \end{cases}
\]  
     Let $C = \left\{ {1,2, \cdots ,i} \right\}$.  Then $\mu \left( C \right) = \lambda $.  Let $S$ be the semistandard $\lambda $-tableau of type $\mu \left( C \right)$ where $S_{j,l}  = j\,,\,1 \leqslant l \leqslant \lambda _j \,,\,1 \leqslant j \leqslant m$.  There is only one standard $\lambda $-tableau of type $S$, namely $s = id$, the identity in $\mathfrak{S}_i $.  The orbit $O\left( {C,\mu } \right) = \left\{ C \right\}$, so $\# O\left( {C,\mu } \right) = 1$.  Also, $\phi _C :\bar i \to \bar r$ is the identity $\phi _C (j) = j\,,\,j = 1,2, \cdots ,i$.

     Define $D \in D\left( {i,r} \right)$ by $D = \left\{ {\left\{ j \right\}:1 \leqslant j \leqslant i} \right\}$.  $M$ contains the rook monoid $\Re _r $, so it contains the map $\alpha :\bar r \cup 0 \to \bar r \cup 0$ given by
      $\alpha \left( j \right) =
         \begin{cases}
           j         &\text{if $1 \leqslant j \leqslant i$ }  \\
           0         &\text{if $j > i$ }
         \end{cases} .$  Then $D = D_\alpha   \in D\left( {i,M,r} \right)$.  As a composition of $i$, $\mu \left( D \right) = \mu \left( C \right) = \lambda $, so we can take $T = S$ as a semistandard $\lambda $
-tableau of type $\mu \left( D \right)$.  Then again there is only one standard $\lambda $-tableau of type $T$, namely $t = id$, the identity in $\mathfrak{S}_i $.  

       We have $O\left( {D,\mu } \right) = \left\{ D \right\}$, $\# O\left( {D,\mu } \right) = 1$,  $\psi _D \left( j \right) =
            \begin{cases}
               j      &\text{if $1 \leqslant j \leqslant i$ }  \\
               0       &\text{if $j > i$ }
             \end{cases} .$  Then $\psi _D  \circ \phi _C  = id:\bar i \to \bar i$.
 
   We have $\,_S C_T^\lambda   = \,_s C_t^\lambda   = id \cdot r_\lambda   \cdot id = r_\lambda  $, so $\,_S C_T^\lambda   \circ \psi _D  \circ \phi _C  \circ \,_S C_T^\lambda   = r_\lambda   \cdot id \cdot r_\lambda   = r_\lambda   \cdot r_\lambda   = o\left( {\mathfrak{S}_\lambda  } \right)\,r_\lambda  $.  Then writing $b = _{O\left( {C,\mu } \right),S} C_{O(D,\mu ),T}^\lambda   = \phi _C  \cdot \,_S C_T^\lambda   \cdot \psi _D $, compute  $b \cdot b = \phi _C  \cdot \,_S C_T^\lambda   \cdot \psi _D  \cdot \phi _C  \cdot \,_S C_T^\lambda   \cdot \psi _D  = \phi _C  \cdot o\left( {\mathfrak{S}_\lambda  } \right) \cdot r_\lambda   \cdot \psi _D  = o\left( {\mathfrak{S}_\lambda  } \right) \cdot b$.

    Then $b * _L b = o\left( {\mathfrak{S}_\lambda  } \right) \cdot \frac{{n_L (b)}}
{{n_L \left( b \right)n_L (b)}} \cdot b = o\left( {\mathfrak{S}_\lambda  } \right) \cdot \frac{1}
{{n_L \left( b \right)}} \cdot b = \frac{{o\left( {\mathfrak{S}_\lambda  } \right)}}
{{\# O\left( {C,\mu } \right)o\left( {\mathfrak{S}_{\mu \left( C \right)} } \right)}} \cdot b$ where $n_L (b) = n_L \left( {O\left( {C,\mu } \right)} \right) = \# O\left( {C,\mu } \right) \cdot o(\mathfrak{S}_{\mu (C)} )$.  Since $\# O\left( {C,\mu } \right) = 1$ and $\mathfrak{S}_\lambda   = \mathfrak{S}_{\mu \left( C \right)} $, we get $b * _L b = b$.

    Computing the bracket in $S_L \left( {M,k} \right)$ gives $\left\langle {C_{O(D,\mu ),T}^\lambda  \,,\,\,_{O\left( {C,\mu } \right),S} C_{}^\lambda  } \right\rangle  = 1 \ne 0$.  So the bracket is not identically zero in $S_L \left( {M,k} \right)$  and $\lambda  \in \Lambda _0 $ as desired. 

      Similarly, $b * _R b = o\left( {\mathfrak{S}_\lambda  } \right) \cdot \frac{{n_R (b)}}
{{n_R \left( b \right)n_R (b)}} \cdot b = o\left( {\mathfrak{S}_\lambda  } \right) \cdot \frac{1}
{{n_R \left( b \right)}} \cdot b = \frac{{o\left( {\mathfrak{S}_\lambda  } \right)}}
{{\# O\left( {D,\mu } \right)o\left( {\mathfrak{S}_{\mu \left( D \right)} } \right)}} \cdot b$ where $n_R (b) = n_R \left( {O\left( {D,\mu } \right)} \right) = \# O\left( {D,\mu } \right) \cdot o(\mathfrak{S}_{\mu (D)} )$.  Since $\# O\left( {D,\mu } \right) = 1$ and $\mathfrak{S}_\lambda   = \mathfrak{S}_{\mu \left( D \right)} $, we get $b * _R b = b$.

    Computing the bracket in $S_R \left( {M,k} \right)$ again gives $\left\langle {C_{O(D,\mu ),T}^\lambda  \,,\,\,_{O\left( {C,\mu } \right),S} C_{}^\lambda  } \right\rangle  = 1 \ne 0$.  So the bracket is not identically zero in $S_R \left( {M,k} \right)$  and $\lambda  \in \Lambda _0 $ as desired. 

     By corollary \ref{c6.1}, cell algebras with $\Lambda _0  = \Lambda $ are quasi-hereditary.
\end{proof}

\end{document}